%% file: main_csc.tex
\documentclass[%
  onecolumn,
]{mpi2015-cscpreprint}

\usepackage{algorithm, algorithmicx, algpseudocode}
\newcommand{\SComment}[1]{\Comment{{\scriptsize #1}}} % shrink comment size
\usepackage{amsthm} % needed for this template
    \theoremstyle{plain}
    \newtheorem{theorem}{Theorem}
	\newtheorem{corollary}{Corollary}
 	\newtheorem{lemma}{Lemma}

    \theoremstyle{definition}

    \theoremstyle{remark}
	\newtheorem{remark}{Remark}
\usepackage{amsmath, amscd, amssymb, bm, mathrsfs, mathtools}
\usepackage{array, booktabs, makecell} % for column headers
\newcolumntype{P}[1]{>{\centering\arraybackslash}p{#1}}
\newcolumntype{M}[1]{>{\centering\arraybackslash}m{#1}}
\usepackage[american]{babel}
\usepackage{enumerate}
\usepackage{graphicx}
\usepackage{hyperref}
\usepackage{soul}
\usepackage{tikz}
\usepackage[backgroundcolor=violet!60,textcolor=white,linecolor=gray]{todonotes}
\usepackage{xspace}

\input{macros}

\begin{document}

%%%%%%%%%%%%%%%%%%%%%%%%%%%%%%%%%%%%%%%%%%%%%%%%%%%%%%%%%%%%%%%%%%%%%%%%%%%%%%%%
% PAPER INFORMATION.                                                           %
%%%%%%%%%%%%%%%%%%%%%%%%%%%%%%%%%%%%%%%%%%%%%%%%%%%%%%%%%%%%%%%%%%%%%%%%%%%%%%%%

\title{Reorthogonalized Pythagorean variants of block classical Gram-Schmidt}

\author[$\ast$]{Erin Carson}
\affil[$\ast$]{Department of Numerical Mathematics, Faculty of Mathematics and Physics, Charles University, Sokolovsk\'{a} 49/83, 186 75 Praha 8, Czechia\authorcr
  \email{\{carson, oktay\}@karlin.mff.cuni.cz}, \email{yuxin.ma@matfyz.cuni.cz}, \orcid{0000-0001-9469-7467}, \orcid{0000-0003-0761-2184}, \orcid{0000-0002-2860-0134}}
  
\author[$\dagger$,$\ddagger$]{Kathryn Lund}
\affil[$\dagger$]{Computational Methods in Systems and Control Theory, Max Planck Institute for Dynamics of Complex Technical Systems, Sandtorstr.\ 1, 39106 Magdeburg, Germany}

\affil[$\ddagger$]{Computational Mathematics Theme, Building R71, STFC Rutherford Appleton Laboratory, Harwell Oxford, Didcot, Oxfordshire, OX11 0QX, United Kingdom\authorcr
    \email{kathryn.lund@stfc.ac.uk}, \orcid{0000-0001-9851-6061}}

\author[$\ast$]{Yuxin Ma}

\author[$\ast,\S$]{Eda Oktay}
\affil[$\S$]{Department of Mathematics, Chemnitz University of Technology, Reichenhainer Str. 41, 09126 Chemnitz, Germany\authorcr}
  
\shortauthor{E. Carson, K. Lund, Y. Ma, and E. Oktay}
  
\keywords{
	Gram-Schmidt algorithm, low-synchronization, communication-avoiding, mixed precision, multiprecision, loss of orthogonality, stability
}

\msc{
	65-04, 65F25, 65G50, 65Y20
}
  
\abstract{
    Block classical Gram-Schmidt (BCGS) is commonly used for orthogonalizing a set of vectors $X$ in distributed computing environments due to its favorable communication properties relative to other orthogonalization approaches, such as modified Gram-Schmidt or Householder.  However, it is known that BCGS (as well as recently developed low-synchronization variants of BCGS) can suffer from a significant loss of orthogonality in finite-precision arithmetic, which can contribute to instability and inaccurate solutions in downstream applications such as $s$-step Krylov subspace methods. A common solution to improve the orthogonality among the vectors is reorthogonalization.  Focusing on the ``Pythagorean" variant of BCGS, introduced in [E.~Carson, K.~Lund, \& M.~Rozlo\v{z}n\'{i}k. \emph{SIAM J.~Matrix Anal.~Appl.} 42(3), pp.~1365--1380, 2021], which guarantees an $O(\varepsilon)\kappa^2(X)$ bound on the loss of orthogonality as long as $O(\varepsilon)\kappa^2(X)<1$, where $\varepsilon$ denotes the unit roundoff, we introduce and analyze two reorthogonalized Pythagorean BCGS variants. These variants feature favorable communication properties, with asymptotically two synchronization points per block column, as well as an improved $O(\varepsilon)$ bound on the loss of orthogonality.  Our bounds are derived in a general fashion to additionally allow for the analysis of mixed-precision variants.  We verify our theoretical results with a panel of test matrices and experiments from a new version of the \texttt{BlockStab} toolbox.
}

\novelty{
	Bounds on the loss of orthogonality are proven for two variants of reorthogonalized block classical Gram-Schmidt, including new variants with asymptotically two synchronization points per block vector.  These bounds are important for the design of scalable, iterative solvers in high-performance computing.  We also examine mixed-precision variants of these methods.
}

\maketitle
%%%%%%%%%%%%%%%%%%%%%%%%%%%%%%%%%%%%%%%%%%%%%%%%%%%%%%%%%%%%%%%%%%%%%%%%%%%%%%%%
% PAPER CONTENT.                                                               %
%%%%%%%%%%%%%%%%%%%%%%%%%%%%%%%%%%%%%%%%%%%%%%%%%%%%%%%%%%%%%%%%%%%%%%%%%%%%%%%%

% Body
\input{sec_1_intro}
\input{sec_2_pip_variants}
\input{sec_3_mp}
\input{sec_4_experiments}
\input{sec_5_conclusions}

\section*{Acknowledgments}%
\addcontentsline{toc}{section}{Acknowledgments}
\input{acknowledgements}
\input{funding}

% \section*{Statement of contributions}%
% \addcontentsline{toc}{section}{Statement of contributions}
% \input{contributions}
%%%%%%%%%%%%%%%%%%%%%%%%%%%%%%%%%%%%%%%%%%%%%%%%%%%%%%%%%%%%%%%%%%%%%%%%%%%%%%%%
% *** REFERENCES ***                                                           %
%%%%%%%%%%%%%%%%%%%%%%%%%%%%%%%%%%%%%%%%%%%%%%%%%%%%%%%%%%%%%%%%%%%%%%%%%%%%%%%%

% Per MPI's rules, the bib file has to follow a certain protocol... Please alert KLund any time you want to add a source.  Also, we have to use the MPI pre-print template (unless there's a competing Charles Uni rule), and we will have to submit our software for a code review.
\addcontentsline{toc}{section}{References}
\bibliographystyle{abbrvurl}
\bibliography{BlockStab}

\appendix
\input{appendix}

\end{document}

%% file: macros.tex
\newcommand{\chol}{\texttt{chol}\xspace}

\newcommand{\spR}{\mathbb{R}}				% field of real scalars

\newcommand{\vv}{\bm{v}}

\newcommand{\vC}{\bm{C}}

\newcommand{\vD}{\bm{D}}

\newcommand{\vE}{\bm{E}}

\newcommand{\vG}{\bm{G}}

\newcommand{\vQ}{\bm{Q}}

\newcommand{\vQbar}{\bar{\vQ}}
\newcommand{\vR}{\bm{R}}

\newcommand{\vS}{\bm{S}}

\newcommand{\vT}{\bm{T}}

\newcommand{\vU}{\bm{U}}

\newcommand{\vUbar}{\bar{\vU}}
\newcommand{\vV}{\bm{V}}
\newcommand{\vVbar}{\bar{\vV}}

\newcommand{\vW}{\bm{W}}
\newcommand{\vWbar}{\bar{\vW}}

\newcommand{\vX}{\bm{X}}

\newcommand{\vY}{\bm{Y}}

\newcommand{\vZ}{\bm{Z}}

\newcommand{\Pbar}{\bar{P}}

\newcommand{\Rbar}{\bar{R}}

\newcommand{\Sbar}{\bar{S}}
\newcommand{\Stil}{\widetilde{S}}
\newcommand{\Ttil}{\widetilde{T}}

\newcommand{\Tbar}{\bar{T}}
\newcommand{\Ombar}{\bar{\Omega}}

\newcommand{\Vtil}{\widetilde{V}}

\newcommand{\DD}{\mathcal{D}}
\newcommand{\EE}{\mathcal{E}}

\newcommand{\RR}{\mathcal{R}}
\newcommand{\RRbar}{\bar{\RR}}
\renewcommand{\SS}{\mathcal{S}}
\newcommand{\SSbar}{\bar{\SS}}

\newcommand{\TT}{\mathcal{T}}
\newcommand{\TTbar}{\bar{\TT}}

\newcommand{\bDD}{\bm{\mathcal{D}}}

\newcommand{\bQQ}{\bm{\mathcal{Q}}}
\newcommand{\bQQbar}{\bar{\bQQ}}

\newcommand{\bUU}{\bm{\mathcal{U}}}
\newcommand{\bUUbar}{\bar{\bUU}}
\newcommand{\bVV}{\bm{\mathcal{V}}}

\newcommand{\bXX}{\bm{\mathcal{X}}}
\newcommand{\inv}{{-1}}
\newcommand{\tinv}{{-T}}

\newcommand{\ihalf}{{-1/2}}

\newcommand{\bigO}[1]{\mathcal{O}\left(#1\right)}

\newcommand{\lmin}{\lambda_{\min}}
\newcommand{\sigmin}{\sigma_{\min}}

\newcommand{\sigmax}{\sigma_{\max}}

\newcommand{\norm}[1]{\left\lVert#1\right\rVert}

\newcommand{\normF}[1]{\norm{#1}_{\text{F}}}
\newcommand{\bmat}[1]{\begin{bmatrix}#1\end{bmatrix}}

\makeatletter
\newsavebox{\@brx}
\newcommand{\llangle}[1][]{\savebox{\@brx}{\(\m@th{#1\langle}\)}%
	\mathopen{\copy\@brx\kern-0.5\wd\@brx\usebox{\@brx}}}
\newcommand{\rrangle}[1][]{\savebox{\@brx}{\(\m@th{#1\rangle}\)}%
	\mathclose{\copy\@brx\kern-0.5\wd\@brx\usebox{\@brx}}}
\makeatother
\newcommand{\eps}{\varepsilon}
\newcommand{\epslo}{\eps_{\ell}}
\newcommand{\epshi}{\eps_h}

\newcommand{\DeltaC}{\Delta C}
\newcommand{\DeltaCone}{\DeltaC^{(1)}}
\newcommand{\DeltaCtwo}{\DeltaC^{(2)}}
\newcommand{\DeltaE}{\Delta E}
\newcommand{\DeltaF}{\Delta F}
\newcommand{\DeltaG}{\Delta G}
\newcommand{\DeltaM}{\Delta M}
\newcommand{\DeltaFone}{\DeltaF^{(1)}}
\newcommand{\DeltaFtwo}{\DeltaF^{(2)}}
\newcommand{\DeltaH}{\Delta H}
\newcommand{\DeltaJ}{\Delta J}
\newcommand{\DeltaL}{\Delta L}
\newcommand{\DeltaP}{\Delta P}
\newcommand{\DeltaR}{\Delta R}
\newcommand{\DeltaS}{\Delta S}
\newcommand{\DeltaStil}{\Delta\Stil}

\newcommand{\DeltaTtil}{\Delta\Ttil}
\newcommand{\DeltaOmega}{\Delta\Omega}

\newcommand{\DeltavD}{\Delta\vD}
\newcommand{\DeltavE}{\Delta\vE}
\newcommand{\DeltavG}{\Delta\vG}
\newcommand{\DeltavGone}{\DeltavG^{(1)}}
\newcommand{\DeltavGtwo}{\DeltavG^{(2)}}

\newcommand{\DeltavR}{\Delta\vR}
\newcommand{\DeltavS}{\Delta\vS}
\newcommand{\DeltavT}{\Delta\vT}
\newcommand{\DeltavU}{\Delta\vU}

\newcommand{\DeltavV}{\Delta\vV}
\newcommand{\DeltavW}{\Delta\vW}
\newcommand{\DeltavX}{\Delta\vX}

\newcommand{\DeltaDD}{\Delta\DD}
\newcommand{\DeltaDDTS}{\DeltaDD_{TS}}
\newcommand{\DeltaEE}{\Delta\EE}

\newcommand{\DeltabDD}{\Delta\bDD}
\newcommand{\DeltabDDUS}{\DeltabDD_{US}}
\newcommand{\DeltabDDQT}{\DeltabDD_{QT}}

\newcommand{\omegaU}{\omega_{U}}
\newcommand{\omegaQ}{\omega_{Q}}
\newcommand{\deltacholone}{\delta_{\chol_1\xspace}}
\newcommand{\deltachol}{\delta_{\chol\xspace}}
\newcommand{\deltacholtwo}{\delta_{\chol_2\xspace}}
\newcommand{\deltaQ}{\delta_{Q}}
\newcommand{\deltaQS}{\delta_{QS}}
\newcommand{\deltaQT}{\delta_{QT}}
\newcommand{\deltaQU}{\delta_{Q^TU}}
\newcommand{\deltaQX}{\delta_{Q^TX}}
\newcommand{\deltaSS}{\delta_{S^TS}}
\newcommand{\deltaTS}{\delta_{TS}}
\newcommand{\deltaTT}{\delta_{T^TT}}
\newcommand{\deltaU}{\delta_{U}}
\newcommand{\deltaUS}{\delta_{US}}
\newcommand{\deltaUU}{\delta_{U^TU}}
\newcommand{\deltaX}{\delta_{X}}
\newcommand{\deltaXX}{\delta_{X^TX}}
\newcommand{\deltaRR}{\delta_{R^TR}}
\newcommand{\deltaQR}{\delta_{QR}}

\newcommand{\rhomax}{\rho_{\max}}
\newcommand{\ximax}{\xi_{\max}}
\newcommand{\default}{\texttt{default}\xspace}
\newcommand{\monomial}{\texttt{monomial}\xspace}

\newcommand{\glued}{\texttt{glued}\xspace}
\newcommand{\piled}{\texttt{piled}\xspace}

\newcommand{\MP}{\mbox{\normalfont\tiny\sffamily MP}}
\newcommand{\BCGS}{\texttt{BCGS}\xspace}	% block classical Gram-Schmidt
\newcommand{\BCGSPIP}{\hyperref[alg:PIP]{\texttt{BCGS-PIP}}\xspace}	% block classical Gram-Schmidt, PIP
\newcommand{\BCGSPIPRO}{\hyperref[alg:PIP+]{\texttt{BCGS-PIP+}}\xspace}	% block classical Gram-Schmidt, PIP with reorthogonalization
\newcommand{\BCGSPIPIRO}{\hyperref[alg:PIPI+]{\texttt{BCGS-PIPI+}}\xspace}	% block classical Gram-Schmidt, PIP with inner reorthogonalization
\newcommand{\BCGSPIPMP}{\hyperref[alg:PIPMP]{\texttt{BCGS-PIP$^{\MP}$}}\xspace}	% mixed-precision block classical Gram-Schmidt, PIP
\newcommand{\BCGSPIPROMP}{\hyperref[alg:PIP+MP]{\texttt{BCGS-PIP+$^{\MP}$}}\xspace}	% mixed-precision block classical Gram-Schmidt, PIP with reorthogonalization
\newcommand{\BCGSPIPIROMP}{\hyperref[alg:PIPI+MP]{\texttt{BCGS-PIPI+$^{\MP}$}}\xspace}	% mixed-precision block classical Gram-Schmidt, PIP with inner reorthogonalization

\newcommand{\IO}[1]{\texttt{IO}\left(#1\right)}	% intra-orthogonalization routine
\newcommand{\IOnoarg}{\texttt{IO}\xspace}	% same, but without a mandatory argument
\newcommand{\HouseQR}{\texttt{HouseQR}\xspace}	% QR via Householder reflections
\newcommand{\TSQR}{\texttt{TSQR}\xspace}	% Tall-Skinny QR
\newcommand{\CholQR}{\texttt{CholQR}\xspace}	% Cholesky QR
\newcommand{\ShCholQRRORO}{\texttt{ShCholQR++}\xspace}	% shifted Cholesky QR with two stages of reorthogonalization

\usepackage{xcolor}
\definecolor{plotblue}{RGB}{0, 113.9850, 188.9550}
\definecolor{plotred}{RGB}{216.7500, 82.8750, 24.9900}
\definecolor{plotpurple}{RGB}{125.9700, 46.9200, 141.7800}

%% file: sec_1_intro.tex
\section{Introduction} \label{sec:intro}
Interest in low-synchronization variants of the Gram-Schmidt method has been proliferating recently \cite{BieLTetal22, CarLR21, CarLRetal22, Lun23, OktC23, ThoCRetal23, YamHBetal24, YamTHetal20, Zou23}.  These methods are part of the more general trend of developing communication-reducing orthogonalization routines with the goal of reducing memory movement between levels of the memory hierarchy or nodes on a network, thereby improving scalability in high-performance, and especially exascale, computing \cite{BalCDetal14, Car15, Hoe10}.  In this manuscript, we concentrate on block Gram-Schmidt (BGS) methods, and in particular, on reorthogonalized versions of block classical Gram-Schmidt with Pythagorean inner product (\BCGSPIP) from \cite{CarLR21}, and how they can achieve loss of orthogonality on the order of unit roundoff with only two synchronization points per block of columns.  In related work \cite{CarLMetal24b}, we consider a generalization of BCGS2-type algorithms \cite{Bar24, BarS13}, which have four such synchronization points per block, and we use this generalization to study the stability of a reorthogonalized BGS variant with only one synchronization point per block.

We define a \emph{block vector} $\vX \in \spR^{m \times s}$ with $m \gg s$ as a concatenation of $s$ column vectors, i.e., a tall-skinny matrix.  We are interested in computing an economic QR decomposition for the concatenation of $p$ block vectors
\[
    \bXX = \bmat{ \vX_1 & \vX_2 & \cdots \vX_p } \in \spR^{m \times ps}.
\]
We achieve this via a BGS method that takes $\bXX$ and a block size $s$ as arguments and returns an orthonormal basis $\bQQ \in \spR^{m \times ps}$ along with an upper triangular $\RR \in \spR^{ps \times ps}$ such that $\bXX = \bQQ \RR$.  Both $\bQQ$ and $\RR$ are computed block-wise, meaning that $s$ new columns of $\bQQ$ are generated per iteration, as opposed to just one column at a time.

Blocking or batching data is a known technique for reducing the total number of synchronization points, or \emph{sync points}.  In a distributed setting, we define a sync point as an operation requiring all nodes to send and receive information to and from one other, such as an \texttt{MPI\_Allreduce}.  In an orthogonalization procedure like BGS, block inner products and \emph{intraorthogonalization} routines (which orthogonalize vectors within a block column) like tall-skinny QR require sync points when block vectors are distributed row-wise across nodes.  For the purposes of this manuscript, we will assume that a block inner product $\vX^* \vY$ and an intraorthogonalization \IOnoarg each constitute one sync point.

In addition to reducing sync points, we are also concerned with the stability of BGS, which we measure here in terms of the \emph{loss of orthogonality (LOO)},
\begin{equation} \label{eq:loo}
    \norm{I - \bQQbar^T \bQQbar},
\end{equation}
where $I$ is the $ps \times ps$ identity matrix and $\bQQbar \in \spR^{m \times ps}$ denotes the $\bQQ$ factor computed in floating-point arithmetic.  We take $\norm{\cdot}$ as the induced matrix 2-norm.  As is common in the relevant literature, we regard a rectangular matrix $\bQQ \in \spR^{m \times ps}$ as \emph{orthogonal} when $I - \bQQ^T \bQQ = 0$, meaning not only are the columns of $\bQQ$ orthogonal to one another but also each column has norm one.  We regard the term \emph{orthonormal} as synonymous with \emph{orthogonal}.

Orthogonality is important for a number of downstream purposes, especially eigenvalue approximation; see, e.g., \cite{GolV13, TreB97} and sources therein.  Furthermore, having LOO close to working precision simplifies the backward error analysis of Krylov subspace methods like GMRES; see, e.g., the modular backward stability framework by Buttari et al.~\cite{ButHMetal24}.  Orthogonality is also stronger than being well-conditioned, i.e., having $\kappa(\bQQbar) \approx 1$, where $\kappa$ denotes the 2-condition number of a matrix, i.e., the ratio between its largest and smallest singular values.  Indeed, Gram-Schmidt methods frequently lose orthogonality while producing well-conditioned bases.  For all these reasons, we concentrate on LOO as the primary metric of a method's stability.

We will also consider the standard residual
\begin{equation} \label{eq:res}
	\norm{\bQQ \RR - \bXX},
\end{equation}
as well as the Cholesky residual,
\begin{equation} \label{eq:chol_res}
    \norm{\bXX^T \bXX - \RRbar^T \RRbar},
\end{equation}
where $\RRbar$ is the finite precision counterpart of $\RR$.  This latter residual measures how close a BGS method is to correctly computing a Cholesky decomposition of $\bXX^T \bXX$, which can provide insight into the stability pitfalls of a method; see, e.g., \cite{CarLR21, GirLRetal05, SmoBL06}.

In the following section, we summarize \BCGSPIP and results from \cite{CarLR21} and prove stability bounds on two reorthogonalized variants, \BCGSPIPRO and \BCGSPIPIRO, which have $2p$ and $2p-1$ sync points, respectively, or roughly 2 sync points per block vector.  Section~\ref{sec:mp} deals with mixed-precision variants of the new reorthogonalized methods, and Section~\ref{sec:experiments} demonstrates the numerical behavior of all methods using the \texttt{BlockStab} toolbox.  We summarize conclusions and future perspectives in Section~\ref{sec:conclusions}.

A few remarks regarding notation are necessary before proceeding.  Generally, uppercase Roman letters ($R_{ij}, S_{ij}, T_{ij}$) denote $s \times s$ block entries of a $ps \times ps$ matrix, which itself is usually denoted by uppercase Roman script ($\RR, \SS, \TT$).  A block column of such matrices is denoted with MATLAB indexing:
\[
    \RR_{1:k-1,k} = \bmat{
        R_{1,k} \\ R_{2,k} \\ \vdots \\ R_{k-1,k}
    }.
\]
For simplicity, we also abbreviate standard $ks \times ks$ submatrices as $\RR_k := \RR_{1:k,1:k}$.

Bold uppercase Roman letters ($\vQ_k$, $\vX_k$, $\vU_k$) denote $m \times s$ block vectors, and bold, uppercase Roman script ($\bQQ, \bXX, \bUU$) denotes an indexed concatenation of $p$ such vectors.  Standard $m \times ks$ submatrices are abbreviated as
\[
    \bQQ_k := \bQQ_{1:k} =
    \bmat{
        \vQ_1 & \vQ_2 & \cdots & \vQ_k
    }.
\]
Note that when $\vX$ is used on its own throughout the text, it denotes a generic block vector.  When a subscript is added to $\vX$ as $\vX_k$, we are referring to the specific $k$th block vector of the input matrix $\bXX$.  Furthermore, $\bXX_k$ is the concatenation of the first $k$ block vectors of $\bXX$, as defined above.

The function $[\vQ, R] = \IO{\vX}$ denotes an \emph{intraorthogonalization} routine, i.e., a method used to orthogonalize vectors within a generic block vector $\vX$. This can be any number of methods, including Householder, classical Gram-Schmidt, modified Gram-Schmidt, or Cholesky QR.

We use $\eps$ to denote the \emph{unit roundoff} of a chosen \emph{working precision}.  For example, if we use IEEE double precision arithmetic, then $\eps = 2^{-53} \approx 1.1\cdot10^{-16}$, and for IEEE single precision, $\eps = 2^{-24} \approx 6.0\cdot 10^{-8}$.  Throughout the text, we use standard results from \cite{Hig02} (particularly Sections 2.2 and 3.5) for rounding-error analysis.

We also make use of big-O notation like $\bigO{\eps}$, which is essentially $c(m, ps) \eps$,  with $c(m, ps)$ denoting a low-degree polynomial in dimensional constants $m$ and $ps$. To enhance clarity, we employ $\bigO{\eps}$ to disregard the dimensional factor $c(m, ps)$, although this factor can become significant when $m$ and $ps$ are large.

%% file: sec_2_pip_variants.tex
\section{Stability of reorthogonalized variants of \texttt{BCGS-PIP}} \label{sec:pip_variants}
\BCGSPIP (Algorithm~\ref{alg:PIP}) is a corrected version of Block Classical Gram-Schmidt (\BCGS), where the block diagonal entries of the $\RR$ factor are computed via the block Pythagorean theorem \cite{CarLR21}.  This correction stabilizes the algorithm by keeping the relative Cholesky residual \eqref{eq:chol_res} close to the working precision.  However, the overall LOO for \BCGSPIP can be quite high, and the  bound 
\[
    \norm{I - \bQQbar^T \bQQbar} \leq \bigO{\eps} \kappa^2(\bXX)
\]
only holds when $\bigO{\eps} \kappa^2(\bXX) \leq 1$.

Thus as $\kappa(\bXX) \to \eps^\ihalf$, any guarantees on the orthogonality of $\bQQbar$ are lost. 

\begin{algorithm}[htbp!]
	\caption{$[\bQQ, \RR] = \BCGSPIP(\bXX, \IOnoarg)$ \label{alg:PIP}}
	\begin{algorithmic}[1]
		\State{$[\vQ_1, R_{11}] = \IO{\vX_1}$}
		\For{$k = 2,\ldots,p$}
		    \State{$\bmat{\RR_{1:k-1,k} \\ P_k} = \bmat{\bQQ_{k-1} & \vX_k}^T \vX_k$} \label{line:BCGSPIP:innerprod}
		    \State{$R_{kk} = \chol\bigl( P_k - \RR_{1:k-1,k}^T \RR_{1:k-1,k} \bigr)$} \label{line:BCGSPIP:chol}
		    \State{$\vV_k = \vX_k - \bQQ_{k-1} \RR_{1:k-1,k}$} \label{line:BCGSPIP:Vk}
		    \State{$\vQ_k = \vV_k R_{kk}^\inv$} \label{line:BCGSPIP:tri}
		\EndFor
		\State \Return{$\bQQ = [\vQ_1, \ldots, \vQ_p]$, $\RR = (R_{ij})$}
	\end{algorithmic}
\end{algorithm}

A standard strategy for improving $\bQQbar$ is running the Gram-Schmidt procedure twice; see, e.g., \cite{BarS13, Ste08} for analysis of reorthogonalized block variants of \BCGS.  We do not consider \BCGS further in this manuscript; for a detailed analysis and provable bounds on its LOO, see \cite{CarLMetal24b}, which also analyzes low-sync reorthogonalized versions of \BCGS.  In the following two subsections, we propose two reorthogonalized versions of \BCGSPIP and prove bounds for their LOO and residuals.

% =========================================================
\subsection{\texttt{BCGS-PIP+}} \label{sec:bcgs_pip_ro}
We first consider the simple approach of running \BCGSPIP twice in a row.  See Algorithm~\ref{alg:PIP+} for what we call \BCGSPIPRO, where $+$ stands for ``reorthogonalization."  Note that despite how the pseudocode is written, it is not necessary to store two bases in practice: $\bUU$ can be stored in place of $\bXX$ throughout the first step.  Similarly, it is possible to build $\RR$ by replacing $\SS$ gradually in the second step, and there is no need to construct $\TT$ explicitly; cf.\ the second phase of Algorithm~\ref{alg:PIPI+}.  The pseudocode is written in such a way as to simplify the mathematical analysis.

\begin{algorithm}[htbp!]
    \caption{$[\bQQ, \RR] = \BCGSPIPRO(\bXX, \IOnoarg)$ \label{alg:PIP+}}
    \begin{algorithmic}[1]
        \State{$[\bUU, \SS] = \BCGSPIP(\bXX, \IOnoarg)$} \label{line:PIP+:US}
        \State{$[\bQQ, \TT] = \BCGSPIP(\bUU, \IOnoarg)$} \label{line:PIP+:QT}
        \State{$\RR = \TT \SS;$} \label{line:PIP+:TS}
        \State \Return{$\bQQ = [\vQ_1, \ldots, \vQ_p]$, $\RR = (R_{ij})$}
    \end{algorithmic}
\end{algorithm}

Under minimal assumptions, proving an $\bigO{\eps}$ LOO bound for Algorithm~\ref{alg:PIP+} follows directly from the results of \cite{CarLR21}. In particular, we can obtain the following result by applying \cite[Theorem~3.4]{CarLR21} twice.  Note the condition $\kappa(\vX) \leq \kappa(\bXX)$ required for $\IO{\vX}$: when $\vX_k$ is a block vector of $\bXX$, this condition holds by \cite[Corollary~8.6.3]{GolV13}.  Here and elsewhere, it is meant to ensure that $\IOnoarg$ does not handle block vectors with worse conditioning than that of $\bXX$.

\begin{corollary} \label{cor:PIP+:LOO}
    Let $\bXX \in \spR^{m \times ps}$ with $\bigO{\eps} \kappa^2(\bXX) \leq \frac{1}{2}$ and $\bQQbar$ and $\RRbar$ be computed by Algorithm~\ref{alg:PIP+}. Assuming that for all $\vX \in \spR^{m \times s}$ with $\kappa(\vX) \leq \kappa(\bXX)$, $[\vQbar, \Rbar] = \IO{\vX}$ satisfy
    \begin{align*}
        \Rbar^T \Rbar = \vX^T \vX + \DeltaE,   &\quad \norm{\DeltaE} \leq \bigO{\eps} \norm{\vX}^2, \mbox{ and} \\
        \vQbar \Rbar = \vX + \DeltavD,          &\quad \norm{\DeltavD} \leq \bigO{\eps} \bigl( \norm{\vX} + \norm{\vQbar} \norm{\Rbar} \bigr),
    \end{align*}
    then $\bQQbar$ and $\RRbar$ satisfy
    \begin{align}       
        \norm{I - \bQQbar^T \bQQbar} \leq \bigO{\eps},  &\mbox{ and} \label{eq:cor:PIP+:LOO} \\
        \bQQbar \RRbar = \bXX + \DeltabDD,            &\quad \norm{\DeltabDD} \leq \bigO{\eps} \norm{\bXX}. \label{eq:cor:PIP+:res}
    \end{align}
\end{corollary}

We can obtain a further result on the Cholesky residual of \BCGSPIPRO via the following corollary, which follows by applying \cite[Theorem~3.2]{CarLR21} twice.

\begin{corollary} \label{cor:PIP+:chol}
     Let $\bXX \in \spR^{m \times ps}$ with $\bigO{\eps} \kappa^2(\bXX) \leq \frac{1}{2}$ and $\bQQbar$ and $\RRbar$ be computed by Algorithm~\ref{alg:PIP+}. Assuming that for all $\vX \in \spR^{m \times s}$ with $\kappa(\vX) \leq \kappa(\bXX)$, $[\vQbar, \Rbar] = \IO{\vX}$ satisfy
    \begin{align*}
        \Rbar^T \Rbar = \vX^T \vX + \DeltaE,   &\quad \norm{\DeltaE} \leq \bigO{\eps} \norm{\vX}^2, \mbox{ and} \\
        \vQbar \Rbar = \vX + \DeltavD,          &\quad \norm{\DeltavD} \leq \bigO{\eps} \bigl( \norm{\vX} + \norm{\vQbar} \norm{\Rbar} \bigr),
    \end{align*}
    then $\RRbar$ satisfies
    \begin{equation} \label{eq:cor:PIP+:chol_res}
        \RRbar^T \RRbar = \bXX^T \bXX + \DeltaEE, \quad \norm{\DeltaEE} \leq \bigO{\eps} \norm{\bXX}^2.
    \end{equation}
\end{corollary}

For our mixed-precision analysis, it will also be useful to have a generalized formulation of these bounds that do not rely on a specific precision and that also reveal all the sources of error. From standard rounding-error principles \cite{Hig02}, we can write the following bounds for each step of Algorithm~\ref{alg:PIP+} for constants $\deltaUS, \omegaU, \deltaQT, \omegaQ, \deltaTS \in (0,1)$:
\begin{align}
    &\bUUbar \SSbar = \bXX + \DeltabDDUS,
    \quad \norm{\DeltabDDUS} \leq \deltaUS \norm{\bXX}; \label{eq:PIP+:res:US} \\
    &\norm{I - \bUUbar^T \bUUbar} \leq \omegaU \kappa^2(\bXX);
    \label{eq:PIP+:LOO:U} \\
    &\bQQbar \TTbar = \bUUbar + \DeltabDDQT,
    \quad \norm{\DeltabDDQT} \leq \deltaQT \norm{\bUUbar}; \label{eq:PIP+:res:QT} \\
    &\norm{I - \bQQbar^T \bQQbar} \leq \omegaQ \kappa^2(\bUUbar); \mbox{ and}
    \label{eq:PIP+:LOO:Q} \\
    &\RRbar = \TTbar \SSbar + \DeltaDDTS,
    \quad \norm{\DeltaDDTS} \leq \deltaTS \norm{\TTbar} \norm{\SSbar}. \label{eq:PIP+:R=TS}
\end{align}

The following theorem summarizes how each step of Algorithm~\ref{alg:PIP+} influences the LOO and residual per iteration, which will be useful in Section~\ref{sec:mp} when we consider a two-precision variant.  We omit dependencies on $k$ from the $\delta$ constants, meaning each can be regarded as a maximum over all constants stemming from the same step of the algorithm.  We also generally assume that such constants are much smaller than $1$ and drop ``quadratic" terms resulting from their products.

\begin{theorem} \label{thm:PIP+}
    Let $\bXX \in \spR^{m \times ps}$ and suppose $[\bQQbar, \RRbar] = \BCGSPIPRO(\bXX, \IOnoarg)$. Assuming that \eqref{eq:PIP+:res:US}--\eqref{eq:PIP+:R=TS} are satisfied with $\omegaU \kappa^2(\bXX) \leq \frac{1}{2}$ and $\omegaQ \leq \frac{1}{6}$, then $\bQQbar$ and $\RRbar$ satisfy 
    \begin{align}       
        \norm{I - \bQQbar^T \bQQbar} \leq 3 \cdot \omegaQ,
        &\mbox{ and} \label{eq:thm:PIP+:LOO} \\
        \bQQbar \RRbar = \bXX + \DeltabDD,
        &\quad \norm{\DeltabDD} \leq \bigl(\deltaUS + \sqrt{6} \cdot \deltaQT + 9 \cdot \deltaTS\bigr) \norm{\bXX}. \label{eq:thm:PIP+:res}
    \end{align}
\end{theorem}

\begin{proof}
    From~\eqref{eq:PIP+:LOO:U} and the assumption $\omegaU \kappa^2(\bXX) \leq \frac{1}{2}$, we obtain
    \begin{equation} \label{eq:thm:PIP+:normU}
        \norm{\bUUbar}^2 = \norm{\bUUbar^T \bUUbar}
        \leq \norm{I - \bUUbar^T\bUUbar} + \norm{I}
        \leq \frac{3}{2},
    \end{equation}
    and by the perturbation theory of singular values~\cite[Corollary~8.6.2]{GolV13}, we obtain a lower bound on $\sigmin^2(\bUUbar)$, namely, $\sigmin^2(\bUUbar) \geq \frac{1}{2}$. Consequently, we have the following for $\kappa^2(\bUUbar)$:
    \begin{equation} \label{eq:thm:PIP+:kappaU}
        \kappa^2(\bUUbar) = \dfrac{\sigmax^2(\bUU)}{\sigmin^2(\bUU)} \leq 3.
    \end{equation}
    Applying \eqref{eq:thm:PIP+:kappaU} to \eqref{eq:PIP+:LOO:Q} immediately confirms \eqref{eq:thm:PIP+:LOO}.
    
    To prove \eqref{eq:thm:PIP+:res}, we combine \eqref{eq:PIP+:res:US}, \eqref{eq:PIP+:res:QT}, and \eqref{eq:PIP+:R=TS} to arrive at
    \begin{equation} \label{eq:thm:PIP+:QR}
        \bQQbar \RRbar
        = \bXX + \underbrace{\DeltabDDUS + \DeltabDDQT \SSbar + \bQQbar \DeltaDDTS}_{=: \DeltabDD}.
    \end{equation}
    Note that $\norm{\bQQbar} \leq \frac{\sqrt{6}}{2}$ is derived similarly to~\eqref{eq:thm:PIP+:normU}, which together yield
    \begin{equation} \label{eq:thm:PIP+:res_bound}
        \begin{split} 
            \norm{\DeltabDD}
            & \leq \norm{\DeltabDDUS} + \norm{\DeltabDDQT} \norm{\SSbar} + \norm{\bQQ} \norm{\DeltaDDTS} \\
            & \leq \deltaUS \norm{\bXX} + \deltaQT \norm{\SSbar} + \frac{\sqrt{6}}{2} \cdot \deltaTS \norm{\TTbar} \norm{\SSbar}.
        \end{split}
    \end{equation}
    By~\eqref{eq:PIP+:res:US} and \eqref{eq:PIP+:res:QT}, we obtain
    \begin{align*}
        \SSbar^T \SSbar &= \bXX^T \bXX + \bXX^T \DeltabDDUS + \DeltabDDUS \bXX + \SSbar^T \bigl(I - \bUUbar^T \bUUbar\bigr) \SSbar,\\
        \TTbar^T \TTbar &=\bUUbar^T \bUUbar + \bUUbar^T \DeltabDDQT + \DeltabDDQT^T \bUUbar + \TTbar^T \bigl(I - \bQQbar^T \bQQbar\bigr) \TTbar.
    \end{align*}
    Furthermore, from~\eqref{eq:thm:PIP+:normU}, \eqref{eq:thm:PIP+:kappaU}, and the assumptions $\omegaU \kappa^2(\bXX) \leq \frac{1}{2}$ and $\omegaQ \leq \frac{1}{6}$, it follows that
    \begin{align}
        \norm{\SSbar} & \leq \sqrt{2 + 4 \cdot \deltaUS} \norm{\bXX} \leq \sqrt{6} \norm{\bXX}, \label{eq:thm:PIP+:normS} \\
        \norm{\TTbar} & \leq \sqrt{3 + 6 \cdot \deltaQT} \leq 3. \label{eq:thm:PIP+:normT}
    \end{align}
    Substituting \eqref{eq:thm:PIP+:normS} and \eqref{eq:thm:PIP+:normT} into \eqref{eq:thm:PIP+:res_bound} proves \eqref{eq:thm:PIP+:res}.
\end{proof}

% ====================================================================
\subsection{\texttt{BCGS-PIPI+}} \label{sec:bcgs_pipi_ro}
One drawback to Algorithm~\ref{alg:PIP+} is that two for-loops are required to reorthogonalize the entire basis.  In practice, we can combine the for-loops without introducing additional sync points.  The resulting algorithm is denoted  \BCGSPIPIRO, where \texttt{I+} stands for ``inner reorthogonalization", and is provided as Algorithm~\ref{alg:PIPI+}.  Note that, as in Algorithm~\ref{alg:PIP+}, we construct matrices $\SS$, $\TT$, and $\bUU$ throughout the algorithm, although none of these needs to be explicitly formed in practice.

Combining the for-loops and introducing a general \IOnoarg for the first orthogonalization step of Algorithm~\ref{alg:PIPI+} creates new challenges in proving its stability bounds, primarily because the first block vector is no longer reorthogonalized.  We can therefore no longer directly use the results from \cite{CarLR21}, as a stricter condition will have to be placed on the choice of \IOnoarg.  In the following, we will first develop a generalized approach depending on small constants stemming from particular sources of error, similar to Theorem~\ref{thm:PIP+}.  We then conclude the section with an application to the uniform-precision case by inducting over the generalized results.

\begin{algorithm}[htbp!]
    \caption{$[\bQQ, \RR] = \BCGSPIPIRO(\bXX, \IOnoarg)$ \label{alg:PIPI+}}
    \begin{algorithmic}[1]
        \State{$[\vQ_1, R_{11}] = \IO{\vX_1}$} \SComment{$S_{11} = R_{11}$, $\vU_1 = \vQ_1$, $T_{11} = I$}
        \For{$k = 2, \ldots, p$}
            \State{$\bmat{\SS_{1:k-1,k} \\ \Omega_k} = \bmat{\bQQ_{k-1} & \vX_k}^T \vX_k$} \SComment{First \BCGSPIP step} \label{line:PIPI+:vSk}
            \State{$S_{kk} = \chol\bigl( \Omega_k - \SS_{1:k-1,k}^T \SS_{1:k-1,k} \bigr)$}
            \State{$\vV_k = \vX_k - \bQQ_{k-1} \SS_{1:k-1,k}$} \label{line:PIPI+:Vk}
            \State{$\vU_k = \vV_k S_{kk}^\inv$} \label{line:PIPI+:Uk}
            \State{$\bmat{\TT_{1:k-1,k} \\ P_k} = \bmat{\bQQ_{k-1} & \vU_k}^T \vU_k$} \SComment{Second \BCGSPIP step}\label{line:PIPI+:vTk}
            \State{$T_{kk} = \chol\bigl( P_k - \TT_{1:k-1,k}^T \TT_{1:k-1,k} \bigr)$}
            \State{$\vW_k = \vU_k - \bQQ_{k-1} \TT_{1:k-1,k}$}
            \State{$\vQ_k = \vW_k T_{kk}^\inv$} \label{line:PIPI+:Qk}
            \State{$\RR_{1:k-1,k} = \SS_{1:k-1,k} + \TT_{1:k-1,k} S_{kk}$} \SComment{Finalize $\RR$ entries}
            \State{$R_{kk} = T_{kk} S_{kk}$} \label{line:PIPI+:Rkk}
        \EndFor
        \State \Return{$\bQQ = [\vQ_1, \ldots, \vQ_p]$, $\RR = (R_{ij})$}
    \end{algorithmic}
\end{algorithm}

For the following, we assume all quantities are computed by Algorithm~\ref{alg:PIPI+} and that at each iteration $k \in \{2, \ldots, p\}$, there exists $\omega_{k-1} \in (0,1)$ such that
\begin{equation} \label{eq:PIPI+:LOO:k-1}
    \norm{I - \bQQbar_{k-1}^T \bQQbar_{k-1}} \leq \omega_{k-1}.
\end{equation}
Then similarly to~\eqref{eq:thm:PIP+:normU},
\begin{equation} \label{eq:PIPI+:normQ}
    \norm{\bQQbar_{k-1}} \leq \sqrt{1 + \omega_{k-1}} \leq \sqrt{2}.
\end{equation}

We can write the following for intermediate quantities computed by \BCGSPIPIRO at each iteration $k \in \{2, \ldots, p\}$, where we summarize rounding-error bounds via constants $\delta_* \in (0,1)$:
\begin{align}
    \SSbar_{1:k-1,k}			&= \bQQbar_{k-1}^T \vX_k + \DeltavS_k,
    \quad \norm{\DeltavS_k}	\leq \deltaQX \norm{\vX_k}; \label{eq:PIPI+:vSk} \\
    \Ombar_k 					&= \vX_k^T \vX_k + \DeltaOmega_k,
    \quad \norm{\DeltaOmega_k} \leq \deltaXX \norm{\vX_k}^2; \label{eq:PIPI+:Omegak} \\
    \Sbar_{kk}^T\Sbar_{kk} 		&= \Ombar_k - \SSbar_{1:k-1,k}^T \SSbar_{1:k-1,k} + \DeltaFone_k + \DeltaCone_k, \label{eq:PIPI+:Skk_chol} \\
    \norm{\DeltaFone_k}			&\leq \deltaSS \norm{\vX_k}^2, \label{eq:PIPI+:Skk_chol_errbound}
    \quad \norm{\DeltaCone_k} \leq \deltacholone \norm{\vX_k}^2; \\
    \vVbar_k 					&= \vX_k - \bQQbar_{k-1} \SSbar_{1:k-1,k} + \DeltavV_k,
    \quad \norm{\DeltavV_k} \leq \deltaQS\norm{\vX_k};
    \label{eq:PIPI+:Vk} \\
    \vUbar_k \SSbar_{kk}        &= \vVbar_k + \DeltavGone_k,
    \quad \norm{\DeltavGone_k} \leq \deltaU \norm{\vUbar_k} \norm{\SSbar_{kk}} \label{eq:PIPI+:Uk}; \\
    \TTbar_{1:k-1,k}            &= \bQQbar_{k-1}^T \vUbar_k + \DeltavT_k,
    \quad \norm{\DeltavT_k} \leq \deltaQU \norm{\vUbar_k}; \label{eq:PIPI+:vTk} \\
    \Pbar_k                     &= \vUbar_k^T \vUbar_k + \DeltaP_k,
    \quad \norm{\DeltaP_k} \leq \deltaUU \norm{\vUbar_k}^2 \label{eq:PIPI+:Pk}; \\ 
    \Tbar_{kk}^T \Tbar_{kk}     &= \Pbar_k - \TTbar_{1:k-1,k}^T \TTbar_{1:k-1,k} + \DeltaFtwo_k + \DeltaCtwo_k \label{eq:PIPI+:Tkk_chol}, \\
    \norm{\DeltaFtwo_k}         &\leq \deltaTT \norm{\vUbar_k}^2
    \mbox{ and }
    \norm{\DeltaCtwo_k} \leq \deltacholtwo \norm{\vUbar_k}^2; \\
    \vWbar_k                    &= \vUbar_k - \bQQbar_{k-1} \TTbar_{1:k-1,k} + \DeltavW_k,
    \quad \norm{\DeltavW_k} \leq \deltaQT \norm{\vUbar_k}; \mbox{ and}
    \label{eq:PIPI+:Wk} \\
    \vQbar_k \Tbar_{kk}         &= \vWbar_k + \DeltavGtwo_k, \quad \norm{\DeltavGtwo_k} \leq \deltaQ \norm{\vQbar_k} \norm{\Tbar_{kk}}.
    \label{eq:PIPI+:Tkk}
\end{align}
We have applied \eqref{eq:PIPI+:normQ} and dropped quadratic error terms throughout the proofs in this section and in \eqref{eq:PIPI+:vSk}--\eqref{eq:PIPI+:Tkk} to simplify the expressions; namely, $\norm{\bUU_k}$ and $\norm{\bQQ_k}$ are baked into the constants, where $\bUU_k = \bmat{\vU_1 & \vU_2 & \cdots & \vU_k}$. We also again omit explicit dependence on $k$ for readability.  Note as well that each $\DeltaF_k^{(i)}$ denotes the floating-point error from the subtraction of the inner product from the previously computed $\Ombar_k$ or $\Pbar_k$, and $\DeltaC_k^{(i)}$ denotes the error from the Cholesky factorization of that result.

In order to apply the Cholesky factorization and obtain \eqref{eq:PIPI+:Skk_chol} and \eqref{eq:PIPI+:Tkk_chol} in the first place, we need that $\Ombar_k - \SSbar_{1:k-1,k}^T \SSbar_{1:k-1,k} + \DeltaFone_k$ and $\Pbar_k - \TTbar_{1:k-1,k}^T \TTbar_{1:k-1,k} + \DeltaFtwo_k$ are symmetric positive definite. We show that $\Ombar_k - \SSbar_{1:k-1,k}^T \SSbar_{1:k-1,k} + \DeltaFone_k$ is positive definite using the following lemma; symmetry is already clear.

\begin{lemma} \label{lem:PIPI+:SPD}
    Fix $k \in \{2, \ldots, p\}$ and suppose that \eqref{eq:PIPI+:LOO:k-1}--\eqref{eq:PIPI+:Omegak} are satisfied, along with
    \begin{equation} \label{eq:lem:PIPI+:SPD:res}
         \bQQbar_{k-1} \RRbar_{k-1}
         = \bXX_{k-1} + \DeltabDD_{k-1}, \quad \norm{\DeltabDD_{k-1}} \leq \deltaX \norm{\bXX_{k-1}}.
    \end{equation}
    Assume
     \begin{equation} \label{eq:lem:PIPI+:SPD:kappa}
        \begin{split}
            \bigl(\deltaX + 2 \cdot \omega_{k-1} + \deltaSS + \deltaXX + 2 \sqrt{2} \cdot \deltaQX \bigr) \kappa^2(\vX_k) < 1.
        \end{split}
    \end{equation}
    Then it holds that
    \begin{equation}
        \lmin \bigl(\Ombar_k - \SSbar_{1:k-1,k}^T \SSbar_{1:k-1,k} + \DeltaFone_k \bigr) > 0.
    \end{equation}
\end{lemma}

\begin{proof}
    By \eqref{eq:PIPI+:vSk} and \eqref{eq:PIPI+:Omegak}, and dropping quadratic error terms, we have
    \begin{equation} \label{eq:lem:PIPI+:SPD:Skk^2}
        \begin{split}
            \Ombar_k & - \SSbar_{1:k-1,k}^T \SSbar_{1:k-1,k} + \DeltaFone_k  \\
            & \quad = \vX_k^T (I - \bQQbar_{k-1} \bQQbar_{k-1}^T) (I - \bQQbar_{k-1} \bQQbar_{k-1}^T) \vX_k + \DeltaG_k,
        \end{split}
    \end{equation}
    where 
    \begin{equation*}
        \begin{split}
            \DeltaG_k
            & := \DeltaOmega_k + \DeltaFone_k - \vX_k^T \bQQbar_{k-1} \DeltavS_k - \DeltavS_k^T \bQQbar_{k-1}^T \vX_k \\
            & \quad + \vX_k^T \bQQbar_{k-1} \bQQbar_{k-1}^T(I - \bQQbar_{k-1} \bQQbar_{k-1}^T) \vX_k
        \end{split}
    \end{equation*}
    satisfies the following, thanks to
    \begin{equation*}
        \begin{split}
            \norm{\vX_k^T \bQQbar_{k-1} \bQQbar_{k-1}^T(I - \bQQbar_{k-1} \bQQbar_{k-1}^T) \vX_k} & = \norm{\vX_k^T \bQQbar_{k-1} (I - \bQQbar_{k-1}^T \bQQbar_{k-1}) \bQQbar_{k-1}^T \vX_k} \\
            & \leq \norm{\vX_k}^2 \norm{\bQQbar_{k-1}}^2 \norm{I - \bQQbar_{k-1}^T \bQQbar_{k-1}}
        \end{split}
    \end{equation*}
    and \eqref{eq:PIPI+:normQ}--\eqref{eq:PIPI+:Omegak}:
    \begin{equation} \label{eq:lem:PIPI+:SPD:Omega-STS}
        \begin{split}
            \norm{\DeltaG_k}
            & \leq \norm{\vX_k}^2 \norm{\bQQbar_{k-1}}^2 \norm{I - \bQQbar_{k-1}^T \bQQbar_{k-1}}
            + \deltaSS \norm{\vX_k}^2 \\
            & \quad + \deltaXX \norm{\vX_k}^2 
            + 2 \sqrt{1 + \omega_{k-1}}\deltaQX\norm{\vX_k}^2 \\
            & \leq \bigl((1 + \omega_{k-1}) \omega_{k-1} + \deltaSS + \deltaXX
            + 2\sqrt{1 + \omega_{k-1}} \deltaQX \bigr) \norm{\vX_k}^2.
        \end{split}
    \end{equation}
    
    Using \eqref{eq:lem:PIPI+:SPD:Skk^2} and \eqref{eq:lem:PIPI+:SPD:Omega-STS}, we obtain
    \begin{equation} \label{eq:lem:PIP:SPD:sigmin-omega-STS}
        \begin{split}
            \lmin\bigl(\Ombar_k & - \SSbar_{1:k-1,k}^T \SSbar_{1:k-1,k} + \DeltaFone_k\bigr) \\
            & \geq \lmin(\vX_k^T (I - \bQQbar_{k-1} \bQQbar_{k-1}^T) (I - \bQQbar_{k-1} \bQQbar_{k-1}^T) \vX_k) - \norm{\DeltaG_k}  \\
            & = \sigmin^2((I - \bQQbar_{k-1} \bQQbar_{k-1}^T) \vX_k) - \norm{\DeltaG_k}  \\
            & \geq \sigmin^2((I - \bQQbar_{k-1} \bQQbar_{k-1}^T) \vX_k)- \bigl((1 + \omega_{k-1})\omega_{k-1} + \deltaSS \\
            & \quad + \deltaXX + 2 \sqrt{1 + \omega_{k-1}} \deltaQX\bigr)\norm{\vX_k}^2.
        \end{split}
    \end{equation}
    Then we estimate $\sigmin((I - \bQQbar_{k-1} \bQQbar_{k-1}^T) \vX_k)$. 
    Notice that $(I - \bQQbar_{k-1} \bQQbar_{k-1}^T) \vX_k$ can be written as, by the assumption~\eqref{eq:lem:PIPI+:SPD:res},
    \begin{equation}
        \begin{split}
            \vX_k - \bQQbar_{k-1} \bQQbar_{k-1}^T \vX_k
            & = \vX_k - \bQQbar_{k-1} \RRbar_{k-1} \RRbar_{k-1}^\inv \bQQbar_{k-1}^T \vX_k \\
            & = \bmat{\bXX_{k-1} + \DeltabDD_{k-1}& \vX_k} \vC^T
        \end{split}
    \end{equation}
    with $\vC := \bmat{-(\RRbar_{k-1}^\inv \bQQbar_{k-1}^T \vX_k)^T & I_s} \in \spR^{s \times sk}$, and $I_s$ denoting the $s \times s$ identity matrix.
    Then from the perturbation theory of singular values~\cite[Corollary~8.6.2]{GolV13}, we obtain
    \begin{equation*}
        \begin{split}
            \sigmin( & (I - \bQQbar_{k-1} \bQQbar_{k-1}^T) \vX_k) \\
            & = \sqrt{\min_{\vv \in \spR^{sk}\setminus{\bm{0}}} \left(\frac{\norm{\bmat{\bXX_{k-1} + \DeltabDD_{k-1}& \vX_k} \vC^T \vv}^2}{\norm{\vv}^2} \right)} \\
            & = \sqrt{\min_{\vv \in \spR^{sk}\setminus{\bm{0}}} \left(\frac{\norm{\bmat{\bXX_{k-1} + \DeltabDD_{k-1}& \vX_k} \vC^T \vv}^2}{\norm{\vC^T \vv}^2}
            \frac{\norm{\vC^T \vv}^2}{\norm{\vv}^2} \right)} \\
            & \geq \sqrt{\min_{\vv \in \spR^{sk}\setminus{\bm{0}}}
            \left(\frac{\norm{\bmat{\bXX_{k-1} + \DeltabDD_{k-1} & \vX_k}(\vC^T \vv)}^2}{\norm{\vC^T \vv}^2} \right)
            \min_{\vv \in \spR^{sk}\setminus{\bm{0}}} \left(\frac{\norm{\vC^T \vv}^2}{\norm{\vv}^2} \right)} \\
            & \geq \sigmin(\bmat{\bXX_{k-1} + \DeltabDD_{k-1} & \vX_k}) \\
            & \geq \sigmin(\bXX_k) - \norm{\DeltabDD_{k-1}} \\
            & \geq \sigmin(\bXX_k) - \deltaX \norm{\bXX_{k-1}}.
        \end{split}
    \end{equation*}
    Together with~\eqref{eq:lem:PIP:SPD:sigmin-omega-STS}, it holds that
    \begin{equation}
        \begin{split}
            \lmin\bigl(\Ombar_k & - \SSbar_{1:k-1,k}^T \SSbar_{1:k-1,k} + \DeltaFone_k\bigr) \\
            & \geq \sigmin^2(\bXX_k) - \deltaX \norm{\bXX_k}^2 - \bigl((1 + \omega_{k-1})\omega_{k-1} + \deltaSS \\
            & \quad + \deltaXX + 2 \sqrt{1 + \omega_{k-1}} \deltaQX\bigr)\norm{\vX_k}^2 \\
            & \geq \sigmin^2(\bXX_k) \bigl(1 - \bigl(\deltaX + (1 + \omega_{k-1})\omega_{k-1}
            + \deltaSS \\
            & \quad + \deltaXX + 2 \sqrt{1 + \omega_{k-1}} \deltaQX\bigr) \kappa^2(\vX_k)\bigr).
        \end{split}
    \end{equation}

    By dropping the quadratic error terms and using assumption \eqref{eq:lem:PIPI+:SPD:kappa}, which guarantees
    \[
        \bigl(\deltaX + (1 + \omega_{k-1}) \omega_{k-1} + \deltaSS + \deltaXX + 2 \sqrt{1 + \omega_{k-1}} \deltaQX\bigr) \kappa^2(\vX_k) < 1,
    \]
    we can therefore conclude that $\lmin \bigl(\Ombar_k - \SSbar_{1:k-1,k}^T \SSbar_{1:k-1,k} + \DeltaFone_k\bigr) > 0$.
\end{proof}

From Lemma~\ref{lem:PIPI+:SPD}, we have shown that $\Ombar_k - \SSbar_{1:k-1,k}^T \SSbar_{1:k-1,k} + \DeltaFone_k$ is symmetric positive definite. Before proving that $\Pbar_k - \TTbar_{1:k-1,k}^T \TTbar_{1:k-1,k} + \DeltaFtwo_k$ is also symmetric positive definite, we need to first bound $\norm{\vUbar_k}$. 

\begin{lemma} \label{lem:PIPI+}
    Fix $k \in \{2, \ldots, p\}$ and suppose that \eqref{eq:PIPI+:LOO:k-1}--\eqref{eq:PIPI+:Uk} are satisfied.  Furthermore suppose that $\bQQbar_{k-1}$ and $\RRbar_{k-1}$ satisfy the following:
    \begin{align}
         \bQQbar_{k-1} \RRbar_{k-1}
         &= \bXX_{k-1} + \DeltabDD_{k-1},\quad \norm{\DeltabDD_{k-1}} \leq \deltaX \norm{\bXX_{k-1}}, \label{eq:lem:PIPI+:res}
    \end{align}
    with $\deltaX, \omega_{k-1} \in (0,1)$.  Assume
    \begin{equation} \label{eq:lem:PIPI+:delUkappa}
        8 \cdot \deltaU \kappa^2(\bXX) \leq 1
    \end{equation}
    and
    \begin{equation}  \label{eq:lem:PIPI+:kappa}
        \begin{split}
            & \bigl(\deltaX + 2 \cdot \omega_{k-1} + 2 \cdot \deltaSS + 2 \cdot \deltaXX + 2 \cdot \deltacholone \\
            & + 18 \cdot \deltaQX + 8 \cdot \deltaQS\bigr) \kappa^2(\bXX_k) \leq \frac{1}{2}.
        \end{split}
    \end{equation}
    Then it holds that
    \begin{align}
        \Sbar_{kk}^T \Sbar_{kk}
        & = \vX_k^T \vX_k - \vX_k^T \bQQbar_{k-1} \bQQbar_{k-1}^T \vX_k + \DeltaS_{kk}, \mbox{ and} \label{eq:lem:PIPI+:Skk^2} \\
        \vUbar_k \Sbar_{kk}
        & = \vX_k - \bQQbar_{k-1} \bQQbar_{k-1}^T \vX_k + \DeltavU_k, \mbox{ where} \label{eq:lem:PIPI+:UkSkk} \\
        \norm{\DeltaS_{kk}} & \leq \bigl(\deltaSS + \deltaXX + \deltacholone + 2 \sqrt{2} \cdot\deltaQX \bigr) \norm{\vX_k}^2, \label{eq:lem:PIPI+:DeltaSkk} \\
        \norm{\Sbar_{kk}} & \leq 3 \norm{\vX_k}, \label{eq:lem:PIPI+:Skk} \\
        \norm{\Sbar_{kk}^\inv} &\leq \frac{\sqrt{2}}{\sigmin(\bXX_k)},
        \label{eq:lem:PIPI+:Skkinv}\\
        \norm{\DeltavU_k} & \leq \bigl(\deltaU \norm{\vUbar_k} + \sqrt{2} \cdot\deltaQX + \deltaQS\bigr) \norm{\vX_k}, \label{eq:lem:PIPI+:DeltaUk} \mbox{ and} \\
        \norm{\vUbar_k} & \leq 2. \label{eq:lem:PIPI+:Uk}
    \end{align}
\end{lemma}

\begin{proof}
    From \eqref{eq:PIPI+:normQ}, \eqref{eq:PIPI+:vSk}, and \eqref{eq:PIPI+:Omegak}, we obtain
    \begin{equation} \label{eq:lem:PIPI+:Omegak_norm_vSk_norm}
        \begin{split}
            \norm{\SSbar_{1:k-1,k}}
            & \leq (\sqrt{2} + \deltaQX) \norm{\vX_k} \mbox{ and} \\
            \norm{\Ombar_k}
            & \leq (1 + \deltaXX) \norm{\vX_k}^2.
        \end{split}
    \end{equation}
    Substituting \eqref{eq:PIPI+:vSk}, \eqref{eq:PIPI+:Omegak}, and \eqref{eq:lem:PIPI+:Omegak_norm_vSk_norm} into \eqref{eq:PIPI+:Skk_chol} gives \eqref{eq:lem:PIPI+:Skk^2} with
    \begin{equation*}
        \DeltaS_{kk} := \DeltaOmega_k + \DeltaFone_k + \DeltaCone_k - \vX_k^T \bQQbar_{k-1} \DeltavS_k - \DeltavS_k^T \bQQbar_{k-1}^T \vX_k.
    \end{equation*}
    From the bounds~\eqref{eq:PIPI+:normQ},  \eqref{eq:PIPI+:vSk}, \eqref{eq:PIPI+:Omegak},  \eqref{eq:PIPI+:Skk_chol}, and \eqref{eq:PIPI+:Skk_chol_errbound}, the desired bounds \eqref{eq:lem:PIPI+:DeltaSkk} and \eqref{eq:lem:PIPI+:Skk} then follow immediately.

    To find \eqref{eq:lem:PIPI+:Skkinv}, we rewrite \eqref{eq:lem:PIPI+:Skk^2}  as 
    \begin{equation} \label{eq:lem:PIPI+:Skk^2:1}
		\Sbar_{kk}^T \Sbar_{kk}
		= \vX_k^T (I - \bQQbar_{k-1} \bQQbar_{k-1}^T) (I - \bQQbar_{k-1} \bQQbar_{k-1}^T) \vX_k
		+ \DeltaStil_k,
	\end{equation}
    where $\DeltaStil_k := \vX_k^T \bQQbar_{k-1} \bQQbar_{k-1}^T(I - \bQQbar_{k-1} \bQQbar_{k-1}^T) \vX_k + \DeltaS_{kk}$. Using a similar logic in the proof of Lemma \ref{lem:PIPI+:SPD} together with \eqref{eq:PIPI+:LOO:k-1} and \eqref{eq:PIPI+:normQ} we obtain \eqref{eq:lem:PIPI+:Skkinv} since
    $\norm{\Sbar_{kk}^\inv} = \frac{1}{\sigmin(\Sbar_{kk})}$ and
    \begin{equation} \label{eq:lem:PIPI+:sigmin_Skk}
        \begin{split}
            \sigmin^2(\Sbar_{kk})
            & \geq \sigmin^2(\bXX_k) - \deltaX \norm{\bXX_k}^2 - \bigl(\omega_{k-1} + \deltaSS \\
            & \quad + \deltaXX + \deltacholone + 2 \sqrt{2} \cdot \deltaQX\bigr)\norm{\vX_k}^2.
        \end{split}
    \end{equation}
    Note that the above reasoning makes sense only if 
    \[
        \sigmin^2(\bXX_k) > \bigl(\deltaX + \omega_{k-1}
        + \deltaSS + \deltaXX + \deltacholone + 2 \sqrt{2} \cdot\deltaQX\bigr)\norm{\vX_k}^2;
    \]
    that is, if
    \[
        \bigl(\deltaX + \omega_{k-1}
        + \deltaSS + \deltaXX + \deltacholone + 2 \sqrt{2} \cdot \deltaQX\bigr) \kappa^2(\bXX_k) < 1,
    \]
    which can be guaranteed by the assumption \eqref{eq:lem:PIPI+:kappa}.

    To prove \eqref{eq:lem:PIPI+:DeltaUk} we first note from~\eqref{eq:PIPI+:Vk} and the substitution of \eqref{eq:PIPI+:vSk} that
    \begin{equation} \label{eq:lem:PIPI+:Vk}
        \begin{split}
            \vVbar_k
            & = \vX_k - \bQQbar_{k-1} \bQQbar_{k-1}^T \vX_k + \DeltavV_k, \\
            \norm{\DeltavV_k}
            & \leq \bigl(\sqrt{2} \cdot \deltaQX + \deltaQS \bigr) \norm{\vX_k}.
        \end{split}
    \end{equation}
    Substituting \eqref{eq:lem:PIPI+:Vk} into \eqref{eq:PIPI+:Uk} leads to \eqref{eq:lem:PIPI+:UkSkk}, with
    \[
        \DeltavU_k := \DeltavGone_k + \DeltavV_k
    \]
    and the desired bound \eqref{eq:lem:PIPI+:DeltaUk} satisfied.

    The final bound \eqref{eq:lem:PIPI+:Uk} requires a bit more work. Multiplying $\vUbar_k \Sbar_{kk}$ with its transpose gives
    \begin{equation} \label{eq:lem:PIPI+:UkSkk^2}
        \begin{split}
            (\vUbar_k \Sbar_{kk})^T (\vUbar_k \Sbar_{kk})
            =\,& \vX_k^T \vX_k - 2 \vX_k^T \bQQbar_{k-1} \bQQbar_{k-1}^T \vX_k \\
            & + \vX_k^T \bQQbar_{k-1} \bQQbar_{k-1}^T \bQQbar_{k-1} \bQQbar_{k-1}^T \vX_k + \DeltaH_k \\
            =\,& \vX_k^T \vX_k - \vX_k^T\bQQbar_{k-1} \bQQbar_{k-1}^T \vX_k \\
            & + \vX_k^T \bQQbar_{k-1}(\bQQbar_{k-1}^T \bQQbar_{k-1} - I) \bQQbar_{k-1}^T \vX_k + \DeltaH_k,
        \end{split}
    \end{equation}
    where
    \begin{equation*}
        \DeltaH_k :=
        \vX_k^T \DeltavU_k + \bigl(\DeltavU_k\bigr)^T \vX_k
        + \vX_k^T \bQQbar_{k-1} \bQQbar_{k-1}^T \DeltavU_k
        + \bigl(\DeltavU_k\bigr)^T \bQQbar_{k-1} \bQQbar_{k-1}^T \vX_k.
    \end{equation*}
    Applying~\eqref{eq:PIPI+:normQ} and \eqref{eq:lem:PIPI+:DeltaUk} gives
    \begin{equation} \label{eq:lem:PIPI+:normDeltaHk}
        \norm{\DeltaH_k} \leq \bigl(4 + 2 \cdot \omega_{k-1}\bigr)\bigl(\deltaU \norm{\vUbar_k} + \sqrt{2} \cdot \deltaQX + \deltaQS \bigr) \norm{\vX_k}^2.
    \end{equation}
    Substituting \eqref{eq:lem:PIPI+:Skk^2} into \eqref{eq:lem:PIPI+:UkSkk^2} leads to
    \begin{equation*}
       (\vUbar_k \Sbar_{kk})^T (\vUbar_k \Sbar_{kk})
       = \Sbar_{kk}^T \Sbar_{kk}
       + \vX_k^T \bQQbar_{k-1} (\bQQbar_{k-1}^T \bQQbar_{k-1} - I) \bQQbar_{k-1}^T \vX_k
       + \DeltaH_k - \DeltaS_{kk},
    \end{equation*}
    and then multiplying by $S_{kk}^\tinv$ on the left and $S_{kk}^\inv$ on the right yields
    \begin{equation} \label{eq:lem:PIPI+:Uk^2}
        \begin{split}
           \vUbar_k^T \vUbar_k
           & = I
           + \Sbar_{kk}^\tinv \vX_k^T \bQQbar_{k-1} (\bQQbar_{k-1}^T \bQQbar_{k-1} - I) \bQQbar_{k-1}^T \vX_k \Sbar_{kk}^\inv  \\
           & \quad + \Sbar_{kk}^\tinv (\DeltaH_k - \DeltaS_{kk}) \Sbar_{kk}^\inv.
        \end{split}
    \end{equation}
    Recalling the assumption \eqref{eq:lem:PIPI+:delUkappa}, we note that
    \[
        4\bigl(2 + \omega_{k-1}\bigr) \deltaU \kappa^2(\bXX_k) \leq 4\bigl(2 + \omega_{k-1}\bigr) \deltaU \kappa^2(\bXX) \leq 1.
    \]
    Applying \eqref{eq:PIPI+:LOO:k-1}, \eqref{eq:PIPI+:normQ}, \eqref{eq:lem:PIPI+:normDeltaHk}, \eqref{eq:lem:PIPI+:DeltaSkk}, \eqref{eq:lem:PIPI+:delUkappa} along with \eqref{eq:lem:PIPI+:Skkinv}, taking norms yields the following:
    \begin{equation}\label{eq:lem:PIPI+:Uk_quad}
        \begin{split}
            \norm{\vUbar_k}^2
            & = \norm{\vUbar_k^T\vUbar_k} \\
            & \leq 1 +  \omega_{k-1} \norm{\Sbar_{kk}^\inv}^2 \norm{\vX_k}^2 + 4\bigl(\deltaU \norm{\vUbar_k} + \sqrt{2} \cdot\deltaQX \\
            & \quad + \deltaQS \bigr) \norm{\Sbar_{kk}^\inv}^2 \norm{\vX_k}^2 + \bigl(\deltaSS + \deltaXX + \deltacholone\\
            & \quad  + 2 \sqrt{2} \cdot \deltaQX \bigr) \norm{\Sbar_{kk}^\inv}^2 \norm{\vX_k}^2 \\
            & \leq 1 + 2\bigl( \omega_{k-1} + \deltaSS + \deltaXX + \deltacholone  + 9 \cdot \deltaQX + 4 \cdot \deltaQS \bigr) \kappa^2(\bXX_k) \\
            & \quad + 8 \cdot \deltaU \norm{\vUbar_k} \kappa^2(\bXX_k)\\
            & \leq \frac{3}{2} + \norm{\vUbar_k}.
        \end{split}
    \end{equation}
    Solving the quadratic inequality~\eqref{eq:lem:PIPI+:Uk_quad} gives 
    \begin{equation*}
        \norm{\vUbar_k} \leq 2.
    \end{equation*}
\end{proof}

Now using similar logic as in the proof of Lemma~\ref{lem:PIPI+:SPD}, we can combine the above lemma with \eqref{eq:PIPI+:vTk} and \eqref{eq:PIPI+:Pk} to conclude that $\Pbar_k - \TTbar_{1:k-1,k}^T \TTbar_{1:k-1,k} + \DeltaFtwo_k$ is symmetric positive definite.

The following theorem makes the relationship between each source of error and the bound on the LOO explicit.

\begin{theorem} \label{thm:PIPI+:LOO}
    Fix $k \in \{2, \ldots, p\}$ and assume that \eqref{eq:PIPI+:LOO:k-1}--\eqref{eq:PIPI+:Tkk} are satisfied with
    \begin{equation} \label{eq:thm:PIPI+:LOO:assump}
        \begin{split}
            & 2\bigl(4 \cdot \omega_{k-1} + \deltaSS + \deltaXX + \deltacholone + 7 \cdot \deltaQS
            + 13 \cdot \deltaQX + 14 \cdot \deltaU  \bigr) \kappa^2(\bXX) \\
            & + 4\bigl(\omega_{k-1} + \deltaUU + \deltaTT + \deltacholtwo + 3 \cdot \deltaQU\bigr) \leq \frac{1}{2}.
        \end{split}
    \end{equation}
    Further assume that the assumptions of Lemma~\ref{lem:PIPI+} hold. Then $\bQQbar_k$ satisfies
    \begin{equation} \label{eq:thm:PIPI+:LOO}
        \begin{split}
            \norm{I - \bQQbar_k^T \bQQbar_k}
            & \leq 17 \cdot \omega_{k-1} + 30 \cdot \deltaQU + 40 \cdot \deltaQT \\
            & \quad + 120 \cdot \deltaQ + 8 \cdot \deltaTT + 8 \cdot \deltacholtwo.
        \end{split}
    \end{equation}
\end{theorem}

\begin{proof}
    Writing $\bQQbar_k = \bmat{\bQQbar_{k-1} & \vQbar_k}$, we can look at $I - \bQQbar_k^T \bQQbar_k$ block-by-block:
    \begin{equation} \label{eq:thm:PIPI+:IH_block}
        I - \bQQbar_k^T\bQQbar_k =
        \bmat{
            I - \bQQbar_{k-1}^T\bQQbar_{k-1} & \bQQbar_{k-1}^T\vQbar_k \\
            \vQbar_k^T\bQQbar_{k-1} & I - \vQbar_k^T\vQbar_k
        }.
    \end{equation}
    The induction hypothesis takes care of bounding the upper left block.  For the off-diagonals, by~\eqref{eq:lem:PIPI+:Uk}, \eqref{eq:PIPI+:vTk}, \eqref{eq:PIPI+:Pk}, and Lemma~\ref{lem:PIPI+}, we obtain
    \begin{align}
        \norm{\TTbar_{1:k-1,k}}
        & \leq 2 \bigl(\sqrt{1+\omega_{k-1}} + \deltaQU\bigr), \label{eq:thm:PIPI+:vTk_norm} \\
        \norm{\Pbar_k}
        & \leq 4 \bigl(1 + \deltaUU\bigr). \nonumber
    \end{align}
    Using~\eqref{eq:PIPI+:Tkk_chol}, the assumption~\eqref{eq:thm:PIPI+:LOO:assump}, and dropping the quadratic terms, the following bound holds:
    \begin{equation} \label{eq:thm:PIPI+:Tkk}
        \begin{split}
            \norm{\Tbar_{kk}}
            & \leq \sqrt{4 \bigl(1 + \deltaUU\bigr) + 4 \bigl(\sqrt{1+\omega_{k-1}} + \deltaQU\bigr)^2 + 4 \deltaTT + 4 \deltacholtwo} \\
            & \leq \sqrt{\frac{17}{2}} \leq 3.
        \end{split}
    \end{equation}  
    Using \eqref{eq:PIPI+:vTk} and \eqref{eq:PIPI+:Pk}, we can rewrite \eqref{eq:PIPI+:Tkk_chol} as 
    \begin{equation} \label{eq:thm:PIPI+:Tkk^2}
    \begin{split}
        \Tbar_{kk}^T \Tbar_{kk}
        & = \vUbar_k^T (I - \bQQbar_{k-1} \bQQbar_{k-1}^T) (I - \bQQbar_{k-1} \bQQbar_{k-1}^T) \vUbar_k + \DeltaTtil_k, \\
        \norm{\DeltaTtil_k}
        & \leq 4\bigl(\omega_{k-1} + \deltaUU + \deltaTT + \deltacholtwo + 3 \cdot \deltaQU\bigr).
    \end{split}
    \end{equation}
    To bound $\norm{\Tbar_{kk}^\inv}$ we can use \eqref{eq:thm:PIPI+:Tkk^2} to write
    \begin{equation}\label{eq:thm:PIPI+:sigmin_Tkk:1}
        \begin{split}
            \sigmin^2(\Tbar_{kk})
            & \geq \sigmin^2 \bigl( (I - \bQQbar_{k-1} \bQQbar_{k-1}^T) \vUbar_k \bigr) - \norm{\DeltaTtil_k} \\
            & \geq \sigmin^2 \bigl( (I - \bQQbar_{k-1} \bQQbar_{k-1}^T) \vUbar_k \bigr) \\
            & \quad - 4\bigl(\omega_{k-1} + \deltaUU + \deltaTT + \deltacholtwo + 3 \cdot \deltaQU\bigr).
        \end{split}
    \end{equation}
    By~\eqref{eq:lem:PIPI+:UkSkk}, we find that
    \begin{equation} \label{eq:thm:PIPI+:Uk}
        \vUbar_k = \left((I - \bQQbar_{k-1} \bQQbar_{k-1}^T) \vX_k + \DeltavU_k \right) \Sbar_{kk}^\inv.
    \end{equation}
    Multiplying \eqref{eq:thm:PIPI+:Uk} by $I-\bQQbar_{k-1}\bQQbar_{k-1}^T$ on the left yields
    \begin{equation} \label{eq:thm:PIPI+:proj}
        \begin{split}
            (I - \bQQbar_{k-1} \bQQbar_{k-1}^T) \vUbar_k
            & = \vUbar_k - \bQQbar_{k-1} (I - \bQQbar_{k-1}^T \bQQbar_{k-1}) \bQQbar_{k-1}^T \vX_k \Sbar_{kk}^\inv \\
            & \quad - \bQQbar_{k-1} \bQQbar_{k-1}^T \DeltavU_k \Sbar_{kk}^\inv.
        \end{split}
    \end{equation}
    Define $\DeltavE_k := \bQQbar_{k-1} (I - \bQQbar_{k-1}^T \bQQbar_{k-1})\bQQbar_{k-1}^T\vX_k\Sbar_{kk}^\inv + \bQQbar_{k-1} \bQQbar_{k-1}^T\DeltavU_k\Sbar_{kk}^\inv$. Using \eqref{eq:lem:PIPI+:Skkinv}, we then have
    \begin{equation*}
        \norm{\vX_k} \norm{\Sbar_{kk}^\inv}
        \leq \sqrt{2} \cdot \kappa(\bXX_k).
    \end{equation*}
    Together with the induction hypothesis and~\eqref{eq:lem:PIPI+:Uk}, we can bound $\norm{\DeltavE_k}$ by
    \begin{equation} \label{eq:thm:PIPI+:DeltaEk}
        \norm{\DeltavE_k}
        \leq \sqrt{2} \bigl(\omega_{k-1} + 2 \cdot \deltaU + \sqrt{2} \deltaQX + \deltaQS\bigr) \kappa(\bXX_k).
    \end{equation}
    Via \eqref{eq:lem:PIPI+:Uk^2}, \eqref{eq:lem:PIPI+:Uk_quad}, \eqref{eq:thm:PIPI+:proj}, and \eqref{eq:thm:PIPI+:DeltaEk}, we find lower bounds on some key singular values\footnote{again by \cite[Corollary~8.6.2]{GolV13}.}:
    \begin{equation*}
        \begin{split}
            \sigmin^2(\vUbar_k)
            & \geq 1
            - \norm{\Sbar_{kk}^\tinv \vX_k^T \bQQbar_{k-1} (I - \bQQbar_{k-1}^T \bQQbar_{k-1}) \bQQbar_{k-1}^T \vX_k \Sbar_{kk}^\inv} \\
            & \quad - \norm{\Sbar_{kk}^\tinv (\DeltaH_k - \DeltaS_{kk}) \Sbar_{kk}^\inv} \\
            & \geq 1 - 2\bigl(\omega_{k-1} + \deltaSS + \deltaXX + \deltacholone + 6\sqrt{2} \cdot \deltaQX + 4 \cdot \deltaQS \\
            & \quad + 8 \cdot \deltaU \bigr) \kappa^2(\bXX_k),
        \end{split}
    \end{equation*}
    and
    \begin{equation} \label{eq:thm:PIPI+:sigmin_proj}
        \begin{split}
            \sigmin^2((I - \bQQbar_{k-1} & \bQQbar_{k-1}^T) \vUbar_k) \\
            & \geq \left(\sigmin(\vUbar_k) - \norm{\DeltavE_k} \right)^2 \\
            & \geq \sigmin^2(\vUbar_k) - 2\norm{\DeltavE_k} \norm{\vUbar_k} \\
            & \geq \sigmin^2(\vUbar_k) - 4 \sqrt{2} \bigl(\omega_{k-1} + 2 \cdot \deltaU + \sqrt{2}\deltaQX + \deltaQS\bigr) \kappa(\bXX_k)\\
            & \geq 1 - 2\bigl(4 \cdot \omega_{k-1} + \deltaSS + \deltaXX + \deltacholone + 13 \cdot \deltaQX \\
            & \quad + 7 \cdot \deltaQS + 14 \cdot \deltaU \bigr) \kappa^2(\bXX_k).
        \end{split}
    \end{equation}
    Combining \eqref{eq:thm:PIPI+:sigmin_proj} with \eqref{eq:thm:PIPI+:Tkk^2}, \eqref{eq:thm:PIPI+:sigmin_Tkk:1}, and using the fact that $\norm{\Tbar_{kk}^\inv}^2 = \frac{1}{\sigmin^2(\Tbar_{kk})}$, we arrive at
    \begin{equation}\label{eq:thm:PIPI+:Tkkinv_norm}
        \norm{\Tbar_{kk}^\inv} = \sqrt{\norm{\Tbar_{kk}^\inv}^2} \leq \sqrt{2}.
    \end{equation}
    With similar logic as in the derivation of \eqref{eq:PIPI+:Uk}, \eqref{eq:PIPI+:Tkk} gives
    \begin{equation} \label{eq:thm:PIPI+:Qk_CLR21}
        \vQbar_k \Tbar_{kk} = \vWbar_k + \DeltavGtwo_k,
        \quad \norm{\DeltavGtwo_k} \leq \deltaQ \norm{\vQbar_k} \norm{\Tbar_{kk}} \leq 6 \cdot \deltaQ,
    \end{equation}
    where we have applied~\eqref{eq:PIPI+:normQ} and \eqref{eq:thm:PIPI+:Tkk} and simplified the bound.  Multiplying \eqref{eq:thm:PIPI+:Qk_CLR21} by $\bQQbar_{k-1}^T$ on the left yields
    \begin{equation} \label{eq:appen:qqt}
        \bQQbar_{k-1}^T \vQbar_k \Tbar_{kk} = \bQQbar_{k-1}^T\vWbar_k + \bQQbar_{k-1}^T \DeltavGtwo_k.
    \end{equation}
    Then we substitute \eqref{eq:PIPI+:Wk} into \eqref{eq:appen:qqt}, multiply both sides by $\Tbar_{kk}^\inv$, and use \eqref{eq:PIPI+:vTk} to obtain
    \begin{equation*}
        \bQQbar_{k-1}^T\vQbar_k
        = (I - \bQQbar_{k-1}^T\bQQbar_{k-1}) \TT_{1:k-1,k} \Tbar_{kk}^\inv
        + \left(\DeltavT_k + \bQQbar_{k-1}^T(\DeltavGtwo_k + \DeltavW_k) \right) \Tbar_{kk}^\inv.
    \end{equation*}
    Combining \eqref{eq:thm:PIPI+:LOO}, \eqref{eq:lem:PIPI+:Uk}, Lemma~\ref{lem:PIPI+}, and bounds \eqref{eq:PIPI+:vTk}, \eqref{eq:thm:PIPI+:Tkkinv_norm}, \eqref{eq:thm:PIPI+:Qk_CLR21}, and \eqref{eq:PIPI+:Wk} leads to
    \begin{equation} \label{eq:thm:PIPI+:off_diag}
        \begin{split} 
            \norm{\bQQbar_{k-1}^T\vQbar_k}
            & \leq 4 \cdot \omega_{k-1} + 2 \sqrt{2} \cdot \deltaQU + 12 \cdot \deltaQ + 4 \cdot \deltaQT.
        \end{split}
    \end{equation}

    To bound the bottom right entry of \eqref{eq:thm:PIPI+:IH_block}, we combine \eqref{eq:PIPI+:Tkk_chol} and \eqref{eq:PIPI+:Tkk} and note that
    \begin{equation} \label{eq:thm:PIPI+:Hk}
        \begin{split}
            \Tbar_{kk}^T (I - \vQbar_k^T\vQbar_k) \Tbar_{kk}
            & = \Tbar_{kk}^T \Tbar_{kk} - \Tbar_{kk}^T \vQbar_k^T \vQbar_k \Tbar_{kk} \\
            & = \Pbar_k - \TTbar_{1:k-1,k}^T \TTbar_{1:k-1,k} + \DeltaFtwo_k + \DeltaCtwo_k \\
            & \quad - (\vWbar_k - \DeltavGtwo_k)^T (\vWbar_k - \DeltavGtwo_k) \\
            & = \Pbar_k - \TTbar_{1:k-1,k}^T \TTbar_{1:k-1,k} - \vWbar_k^T \vWbar_k \\
            & \quad - \underbrace{\vWbar_k^T \DeltavGtwo_k - (\DeltavGtwo_k)^T \vWbar_k + \DeltaFtwo_k + \DeltaCtwo_k}_{=: \DeltaL_k}.
        \end{split}
    \end{equation}
    Further substituting \eqref{eq:PIPI+:Pk} and \eqref{eq:PIPI+:Wk} into \eqref{eq:thm:PIPI+:Hk} simplifies to
    \begin{equation} \label{eq:thm:PIPI+:Jk}
        \begin{split}
            \Tbar_{kk}^T (I - \vQbar_k^T\vQbar_k) \Tbar_{kk}
            & = - \TTbar_{1:k-1,k}^T \TTbar_{1:k-1,k} + \vUbar_k^T \bQQbar_{k-1} \TTbar_{1:k-1,k} \\
            & \quad + \TTbar_{1:k-1,k}^T \bQQbar_{k-1}^T \vUbar_k - \TTbar_{1:k-1,k}^T \bQQbar_{k-1}^T \bQQbar_{k-1} \TTbar_{1:k-1,k} \\
            & \quad - \DeltaJ_k - \DeltaL_k, 
        \end{split}
    \end{equation}
    where $\DeltaJ_k := \bigl(\vUbar_k - \bQQbar_{k-1} \TTbar_{1:k-1,k}\bigr)^T \DeltavW_k - \bigl(\DeltavW_k\bigr)^T \bigl(\vUbar_k - \bQQbar_{k-1} \TTbar_{1:k-1,k}\bigr)$. One more substitution of \eqref{eq:PIPI+:vTk} into \eqref{eq:thm:PIPI+:Jk} and further simplification yields
    \begin{equation} \label{eq:thm:PIPI+:Tkk*LOO*Tkk}
        \begin{split}
            \Tbar_{kk}^T (I - \vQbar_k^T\vQbar_k) \Tbar_{kk}
            & = \TTbar_{1:k-1,k}^T \bigl(I - \bQQbar_{k-1}^T \bQQbar_{k-1} \bigr) \TTbar_{1:k-1,k} \\
            & \quad - \bigl(\DeltavT_k \bigr)^T \TTbar_{1:k-1,k} - \TTbar_{1:k-1,k}^T \DeltavT_k - \DeltaJ_k - \DeltaL_k.
        \end{split}
    \end{equation}
    Multiplying \eqref{eq:thm:PIPI+:Tkk*LOO*Tkk} by $\Tbar_{kk}^\tinv$ on the left and $\Tbar_{kk}^\inv$ on the right, taking norms, and applying nearly all previous bounds along with \eqref{eq:PIPI+:LOO:k-1} leads to
    \begin{equation} \label{eq:thm:PIPI+:bottom_right}
        \begin{split}
            \norm{I - \vQbar_k^T \vQbar_k}
            & \leq 8 \cdot \omega_{k-1} + 24 \cdot \deltaQU + 32 \cdot \deltaQT \\
            & \quad + 96 \cdot \deltaQ + 8 \cdot \deltaTT + 8 \cdot \deltacholtwo.
        \end{split}
    \end{equation}
    
    Finally, using \eqref{eq:thm:PIPI+:IH_block} along with \eqref{eq:PIPI+:LOO:k-1} and the bounds \eqref{eq:thm:PIPI+:off_diag} and \eqref{eq:thm:PIPI+:bottom_right}, we see that
    \begin{equation*}
        \begin{split}
            \norm{I - \bQQbar_k^T\bQQbar_k}
            & = \norm{
                \bmat{
                    I - \bQQbar_{k-1}^T\bQQbar_{k-1} & \bQQbar_{k-1}^T\vQbar_k \\
                    \vQbar_k^T\bQQbar_{k-1} & I - \vQbar_k^T\vQbar_k
                }
            } \\
            & \leq \norm{
                \bmat{
                    \norm{I-\bQQbar_{k-1}^T\bQQbar_{k-1}} & \norm{\bQQbar_{k-1}^T\vQbar_k} \\
                    \norm{\vQbar_k^T\bQQbar_{k-1}} & \norm{I-\vQbar_k^T\vQbar_k}
                }
            } \\
            & \leq \normF{
                \bmat{
                    \norm{I-\bQQbar_{k-1}^T\bQQbar_{k-1}} & \norm{\bQQbar_{k-1}^T\vQbar_k} \\
                    \norm{\vQbar_k^T\bQQbar_{k-1}} & \norm{I-\vQbar_k^T\vQbar_k}
                }
            } \\
            & \leq \norm{I-\bQQbar_{k-1}^T \bQQbar_{k-1}} + 2 \norm{\bQQbar_{k-1}^T \vQbar_k}
            + \norm{I - \vQbar_k^T \vQbar_k} \\
            & \leq 17 \cdot \omega_{k-1} + 30 \cdot \deltaQU
            + 40 \cdot \deltaQT + 120 \cdot \deltaQ + 8 \cdot \deltaTT + 8 \cdot \deltacholtwo,
        \end{split}
    \end{equation*}
    where we have used \cite[P.15.50]{GarH17} as in \cite[Theorem~3.1]{CarLR21}.
\end{proof}

If we assume a uniform working precision with unit roundoff $\eps$ and apply standard floating-point point analysis~\cite{Hig02} to Theorem~\ref{thm:PIPI+:LOO}, it is not hard to show that
\[
    \deltaQX, \deltaXX, \deltaSS, \deltaQS, \deltaU, \deltaQU, \deltaUU, \deltaTT, \deltaQT, \deltaQ \leq \bigO{\eps}
\]
and
\[
    \deltacholone, \deltacholtwo \leq \bigO{\eps}.
\]

In contrast to Corollary~\ref{cor:PIP+:LOO}, we must impose a LOO condition on $\IOnoarg$; \HouseQR \cite{Hig02}, \TSQR \cite{MorYZ12}, \texttt{CholQR++} \cite{YamNYetal15}, or \ShCholQRRORO \cite{FukKNetal20} should satisfy the requirement, but only \TSQR could do so and still maintain a single sync point for the \IOnoarg~\cite{BalDGetal15}\footnote{Note that TSQR can be performed as a single reduction, but the resulting $\vQ$ factor is implicitly represented in a tree-based  format; if the Householder representation of the $\vQ$ factor is desired, a second reduction is required.}.  The following corollaries summarize bounds for \BCGSPIPIRO in uniform precision.

\begin{corollary} \label{cor:PIPI+:LOO}
    Assume that $\bigO{\eps} \kappa^2(\bXX) \leq \frac{1}{2}$ and that for all $\vX \in \spR^{m \times s}$ with $\kappa(\vX) \leq \kappa(\bXX)$, $[\vQbar, \Rbar] = \IO{\vX}$ satisfy
    \begin{align*}
        \Rbar^T \Rbar
        & = \vX^T \vX + \DeltaE, \quad \norm{\DeltaE} \leq \bigO{\eps} \norm{\vX}^2 \\ 
        \vQbar \Rbar
        & = \vX + \DeltavD, \quad \norm{\DeltavD} \leq \bigO{\eps} \norm{\vX} \mbox{ and} \\
        \norm{I - \vQbar^T \vQbar}
        & \leq \frac{\bigO{\eps}}{1 - \bigO{\eps} \kappa^2(\vX)}. 
    \end{align*}
    Assume also that $[\bQQbar, \RR] = \BCGSPIPIRO(\bXX, \IOnoarg)$ and for all $k \in \{3, \ldots, p\}$,
    \begin{equation*}
        \bQQbar_{k-1} \RRbar_{k-1} = \bXX_{k-1} + \DeltabDD_{k-1},
        \quad \norm{\DeltabDD_{k-1}}
        \leq \bigO{\eps} \norm{\bXX_{k-1}}.
    \end{equation*}
    Then for all $k \in \{1, \ldots, p\}$, $\bQQbar_k$ satisfies
    \begin{equation} \label{eq:cor:PIPI+:LOO}
        \norm{I - \bQQbar_k^T \bQQbar_k}
        \leq \frac{\bigO{\eps}}{1 - \bigO{\eps} \kappa^2(\bXX_k)}
        \leq \bigO{\eps}.
    \end{equation}
\end{corollary}

\begin{proof}
    First note that $\bigO{\eps} \kappa^2(\bXX) \leq \frac{1}{2}$ implies that for all $k \in \{1, \ldots, p\}$,\\$\bigO{\eps} \kappa^2(\bXX_k) \leq \frac{1}{2}$. Then for the base case, the assumptions on $\IOnoarg$ directly give
    \[
        \norm{I - \bQQbar_1^T \bQQbar_1}
        = \norm{I - \vQbar_1^T \vQbar_1}
        \leq \frac{\bigO{\eps}}{1 - \bigO{\eps} \kappa^2(\vX_1)}
        = \frac{\bigO{\eps}}{1 - \bigO{\eps} \kappa^2(\bXX_1)},
    \]
    which also means that $\omega_1 \leq \frac{\bigO{\eps}}{1 - \bigO{\eps} \kappa^2(\bXX_1)}$. Now assume that \eqref{eq:cor:PIPI+:LOO} holds for all $j \in \{1,\ldots, k-1\}$, i.e., $\omega_{k-1} \leq \frac{\bigO{\eps}}{1 - \bigO{\eps} \kappa^2(\bXX_{k-1})}$. Noting that the assumption $\bigO{\eps} \kappa^2(\bXX) \leq \frac{1}{2}$ can guarantee~\eqref{eq:thm:PIPI+:LOO:assump}, Theorem~\ref{thm:PIPI+:LOO} proves that \eqref{eq:thm:PIPI+:LOO} holds for $k$.
    
    Note that it seems that we prove that $\norm{I - \bQQbar_k^T \bQQbar_k} \leq \frac{\bigO{\eps}}{1 - \bigO{\eps} \kappa^2(\bXX_{k-1})}$, which is because we omitted $\frac{1}{1 - \bigO{\eps} \kappa^2(\bXX_{k-1})}$ for simplicity using the assumption $\bigO{\eps} \kappa^2(\bXX) \leq \frac{1}{2}$ in~\eqref{eq:thm:PIPI+:Tkkinv_norm}. Replacing~\eqref{eq:thm:PIPI+:Tkkinv_norm} in the proof of Theorem~\ref{thm:PIPI+:LOO} with
    \begin{equation} \label{eq:cor:PIPI+:LOO:Tkkinv_norm}
        \norm{\Tbar_{kk}^\inv}
        = \sqrt{\norm{\Tbar_{kk}^\inv}^2}
        \leq \sqrt{\frac{1}{1 - \bigO{\eps} \kappa^2(\bXX_k)}}
        \leq \frac{1}{1 - \bigO{\eps} \kappa^2(\bXX_k)},
    \end{equation}
    we can draw the conclusion.
\end{proof}

A key takeaway from Corollary~\ref{cor:PIPI+:LOO}, in particular the bound~\eqref{eq:thm:PIPI+:LOO}, is that the LOO of \BCGSPIPIRO depends on the conditioning of $\bXX$.  We have made the rather artificial assumption that $\bigO{\eps} \kappa(\bXX) \leq \frac{1}{2}$.  Of course a constant closer to $1$ could be used instead, and it would become clear that the constant on $\bigO{\eps}$ can grow arbitrarily large the closer we let $\kappa(\bXX)$ get to $\frac{1}{\sqrt{\eps}}$.  In practice, this edge-case behavior can be quite dramatic, which examples in Section~\ref{sec:experiments} demonstrate.

We close this section by bounding both the standard residual \eqref{eq:res} and Cholesky residual~\eqref{eq:chol_res} for Algorithm~\ref{alg:PIPI+} in uniform precision.

\begin{corollary} \label{cor:PIPI+:res}
    Assume that $\bigO{\eps} \kappa^2(\bXX) \leq \frac{1}{2}$ and that for all $\vX \in \spR^{m \times s}$ with $\kappa(\vX) \leq \kappa(\bXX)$, $[\vQbar, \Rbar] = \IO{\vX}$ satisfy
    \begin{align*}
        \vQbar \Rbar &= \vX + \DeltavD, \quad \norm{\DeltavD} \leq \bigO{\eps} \norm{\vX}
        \mbox{ and }
        \norm{I - \vQbar^T\vQbar} \leq \frac{\bigO{\eps}}{1 - \bigO{\eps} \kappa^2(\vX)}. 
    \end{align*}
    Then for $[\bQQbar, \RR] = \BCGSPIPIRO(\bXX, \IOnoarg)$ and all $k \in \{1, \ldots, p\}$,
    \[
        \bQQbar_k \RRbar_k = \bXX_k + \DeltabDD_k,
        \quad \norm{\DeltabDD_k}
        \leq \bigO{\eps}\norm{\bXX_k}.
    \]
\end{corollary}
\begin{proof}
    By the assumption on $\IOnoarg$, we have for $k=1$,
    \[
        \bQQbar_1 \RRbar_1 = \bXX_1 + \DeltabDD_1,
        \quad \norm{\DeltabDD_1}
        \leq \bigO{\eps} \norm{\bXX_1}.
    \]
    Now we assume that for all $j \in \{2, \ldots, k-1\}$, it holds that
    \begin{equation}
        \bQQbar_j \RRbar_j = \bXX_j + \DeltabDD_j,
        \quad \norm{\DeltabDD_j}
        \leq \bigO{\eps} \norm{\bXX_j}.
    \end{equation}
    For $k$, it then follows that
    \begin{equation} \label{eq:cor:PIPI+:res:Delta_Xk}
        \begin{split}
            \DeltabDD_k
            & := \bmat{\DeltabDD_{k-1} & \DeltavX_k} \\
            & = \bQQbar_k \RRbar_k - \bXX_k \\
            & = \bmat{
                \bQQbar_{k-1} \RRbar_{k-1} - \bXX_{k-1}
                & \bQQbar_{k-1} \RRbar_{1:k-1,k} + \vQbar_k \Rbar_{kk} - \vX_k
                }.
        \end{split}
    \end{equation}
    The first element of \eqref{eq:cor:PIPI+:res:Delta_Xk} is taken care of by the induction hypothesis.  As for the second element, we treat $\bQQbar_{k-1} \RRbar_{1:k-1,k}$ and $\vQbar_k \Rbar_{kk}$ separately.  There exists $\DeltavR_k$ such that
    \begin{equation} \label{eq:cor:PIPI+:res:vRk}
    	\RRbar_{1:k-1,k} = \SSbar_{1:k-1,k} + \TTbar_{1:k-1,k} \Sbar_{kk} +\DeltavR_k,
    \end{equation}
    \begin{equation} \label{eq:cor:PIPI+:res:DeltavRk}
        \norm{\DeltavR_k}
        \leq \bigO{\eps}(\norm{\SSbar_{1:k-1,k}} + \norm{\TTbar_{1:k-1,k}} \norm{\Sbar_{kk}})
        \leq \bigO{\eps}\norm{\vX_k},
    \end{equation}
    by~\eqref{eq:lem:PIPI+:Omegak_norm_vSk_norm}, \eqref{eq:thm:PIPI+:vTk_norm}, and Lemma~\ref{lem:PIPI+}.  Combining~\eqref{eq:cor:PIPI+:res:vRk}, \eqref{eq:cor:PIPI+:res:DeltavRk}, and ~\eqref{eq:PIPI+:vSk}, it follows that
    \begin{equation} \label{eq:cor:PIPI+:res:QvRk}
        \begin{split}
            \bQQbar_{k-1} \RRbar_{1:k-1,k}
            & = \bQQbar_{k-1} \left( \SSbar_{1:k-1,k} + \TTbar_{1:k-1,k} \Sbar_{kk} +\DeltavR_k \right)\\
            & = \bQQbar_{k-1}\left( \bQQbar_{k-1}^T \vX_k + \DeltavS_k +\TTbar_{1:k-1,k} \Sbar_{kk} + \DeltavR_k \right)\\
            & = \bQQbar_{k-1} \bQQbar_{k-1}^T \vX_k + \bQQbar_{k-1} \TTbar_{1:k-1,k} \Sbar_{kk} + \bQQbar_{k-1} \left( \DeltavS_k + \DeltavR_k \right),
        \end{split}
    \end{equation}
    with
    \begin{equation*}
    	\norm{\bQQbar_{k-1} \left(\DeltavS_k + \DeltavR_k \right)}
    	\leq \bigO{\eps}\norm{\vX_k}.
    \end{equation*}
    From line~\ref{line:PIPI+:Rkk} of Algorithm~\ref{alg:PIPI+}, \eqref{eq:lem:PIPI+:Skk}, and \eqref{eq:thm:PIPI+:Tkk}, we can write 
    \begin{equation} \label{eq:thm:PIPI+:Rkk}
    	\Rbar_{kk} = \Tbar_{kk} \Sbar_{kk} +\DeltaR_{kk}, \quad \norm{\DeltaR_{kk}} \leq \bigO{\eps} \norm{\vX_k}.
    \end{equation}
    Plugging \eqref{eq:PIPI+:Vk}, \eqref{eq:PIPI+:Uk}, \eqref{eq:PIPI+:Wk}, and \eqref{eq:PIPI+:Tkk} into \eqref{eq:thm:PIPI+:Rkk}, the term $\vQbar_k \Rbar_{kk}$ can be rewritten as
    \begin{equation} \label{eq:cor:PIPI+:res:QRkk}
        \begin{split}
            \vQbar_k \Rbar_{kk}
            & = \vQbar_k \left(\Tbar_{kk} \Sbar_{kk} + \DeltaR_{kk} \right) \\
            & = \vQbar_k \Tbar_{kk} \Sbar_{kk} + \vQbar_k \DeltaR_{kk} \\
            & = \left(\vWbar_k + \DeltavGtwo_k \right) \Sbar_{kk} + \vQbar_k \DeltaR_{kk}\\
            & = \left(\vUbar_k - \bQQbar_{k-1} \TTbar_{1:k-1,k}
            + \DeltavW_k + \DeltavGtwo_k \right) \Sbar_{kk} + \vQbar_k \DeltaR_{kk}\\
            & = \vUbar_k \Sbar_{kk} - \bQQbar_{k-1} \TTbar_{1:k-1,k} \Sbar_{kk}
            + \left(\DeltavW_k + \DeltavGtwo_k\right) \Sbar_{kk} + \vQbar_k \DeltaR_{kk}\\
            & = \vX_k - \bQQbar_{k-1} \bQQbar_{k-1}^T \vX_k
            + \DeltavV_k + \DeltavGone_k - \bQQbar_{k-1} \DeltaS_k - \bQQbar_{k-1}\TTbar_{1:k-1,k} \Sbar_{kk} \\
            & \qquad + \left(\DeltavW_k + \DeltavGtwo_k\right) \Sbar_{kk} + \vQbar_k \DeltaR_{kk},
        \end{split}
    \end{equation}
    with 
    \begin{equation*}
        \norm{\DeltavV_k + \DeltavGone_k - \bQQbar_{k-1} \DeltaS_k + \left(\DeltavW_k + \DeltavGtwo_k\right) \Sbar_{kk} + \vQbar_k \DeltaR_{kk}} \leq \bigO{\eps} \norm{\vX_k}.
    \end{equation*}
    Combining \eqref{eq:cor:PIPI+:res:QvRk}, \eqref{eq:cor:PIPI+:res:QRkk}, and the induction hypothesis, we can conclude the proof.
\end{proof}

\begin{remark}
    Invoking Lemma~\ref{lem:PIPI+} in the proof above is not circular logic, as it is valid to use with the induction hypothesis (residual in the $k-1$st step), and \eqref{eq:thm:PIPI+:vTk_norm}, \eqref{eq:PIPI+:Wk}, and \eqref{eq:PIPI+:Tkk} follow from standard rounding-error bounds together with Lemma~\ref{lem:PIPI+}.
\end{remark}

The bound on the Cholesky residual follows directly from Theorems~\ref{cor:PIPI+:LOO} and \ref{cor:PIPI+:res}.
\begin{corollary} \label{cor:PIPI+:cholres}
    Assume that $\bigO{\eps} \kappa^2(\bXX) \leq \frac{1}{2}$ and for all $k \in \{1, \ldots, p\}$, $\bQQbar_k$ computed by Algorithm~\ref{alg:PIPI+} satisfies
    \begin{align}
        \norm{I - \bQQbar_k^T \bQQbar_k}
        & \leq \frac{\bigO{\eps}}{1 - \bigO{\eps} \kappa^2(\bXX_k)} \mbox{ and} \label{cor:PIPI+:cholres:LOO} \\
        \bXX_k + \DeltabDD_k
        & = \bQQbar_k \RRbar_k,
        \quad \norm{\DeltabDD_k}
        \leq \bigO{\eps} \norm{\bXX_k}. \label{eq:cor:PIPI+:cholres:res}
    \end{align}
    Then for all $k \in \{1, \ldots, p\}$,
    \begin{equation*}
        \RRbar_k^T \RRbar_k = \bXX_k^T\bXX_k + \DeltaEE_k, \quad \norm{\DeltaEE_k} \leq \bigO{\eps} \norm{\bXX_k}^2.
    \end{equation*}
\end{corollary}

\begin{proof}
    From \eqref{eq:cor:PIPI+:cholres:res} we can directly write
    \begin{equation*}
        \RRbar_k^T \bQQbar_k^T \bQQbar_k \RRbar_k = \bXX_k^T \bXX_k + \DeltaM_k, \quad \norm{\DeltaM_k} \leq \bigO{\eps} \norm{\bXX_k}^2,
    \end{equation*}
    where $\DeltaM_k = \bXX_k^T \DeltabDD_k + (\DeltabDD_k)^T\bXX_k + (\DeltabDD_k)^2$. Rearranging terms and applying the assumption \eqref{cor:PIPI+:cholres:LOO} yields 
    \begin{equation} \label{eq:cor:PIPI+:cholres:RRk^2}
        \begin{split}
            \RRbar_k^T \RRbar_k
            & = \bXX_k^T \bXX_k + \underbrace{\DeltaM_k + \RRbar_k^T(I - \bQQbar_k^T \bQQbar_k)\RRbar_k}_{=:\DeltaEE_k}, \\
            & \quad \norm{\DeltaEE_k} \leq \bigO{\eps} \left(\norm{\bXX_k}^2 + \norm{\RRbar_k}^2 \right).
        \end{split}
    \end{equation}
    
    Bounding $\norm{\RRbar_k}$ follows by multiplying \eqref{eq:cor:PIPI+:cholres:res} by $\bQQbar_k^T$ on the left and rearranging terms to arrive at
    \begin{equation*}
        \RRbar_k  = (I - \bQQbar_k^T \bQQbar_k)\RRbar_k + \bQQbar_k^T \bXX_k + \bQQbar_k^T \DeltabDD_k.
    \end{equation*}
    Under the assumptions of this lemma and by \eqref{eq:PIPI+:normQ}, we find
    \begin{equation*}
        \norm{\RRbar_k} \leq \frac{1 + \bigO{\eps}}{1 - \bigO{\eps}} \norm{\bXX_k} \leq \bigO{1} \norm{\bXX_k},
    \end{equation*}
    which, substituted back into \eqref{eq:cor:PIPI+:cholres:RRk^2}, completes the proof.
\end{proof}

\begin{remark}
    A version of Algorithm~\ref{alg:PIPI+} has been independently developed and its performance studied in \cite[Figure~4(b)]{YamHBetal24}; the authors there refer to it as \texttt{BCGS-PIP2}.  We became aware of \cite{YamHBetal24} as we were finishing this manuscript.  The authors provide a high-level discussion of the stability analysis of \texttt{BCGS-PIP2}, but they erroneously conflate \BCGSPIPIRO (Algorithm~\ref{alg:PIPI+}) with \BCGSPIPRO (Algorithm~\ref{alg:PIP+}); furthermore, it is unclear what \IOnoarg initializes their \texttt{BCGS-PIP2}.  Indeed, as our detailed analysis demonstrates, interchanging the for-loops has a nontrivial effect on attainable guarantees for LOO, as well as on conditions for the first \IOnoarg.  In particular, \BCGSPIPIRO requires a stronger \IOnoarg than \BCGSPIPRO to maintain LOO, because the first block vector is not reorthogonalized.  Furthermore, neither their analysis nor ours extends to the ``two-stage" algorithm they present, which is essentially a hybrid of \BCGSPIPRO and \BCGSPIPIRO and allows for a larger block size in the reorthogonalization step.  We leave the stability analysis of this hybrid algorithm to future work.
\end{remark}

\begin{remark}
    An immediate consequence of the analysis in this section is that the proven bounds hold trivially for block size $s=1$, and unfortunately the restriction $\bigO{\eps} \kappa^2(\bXX) \leq \frac{1}{2}$ cannot be alleviated; indeed, the size of $s$ only affects the (hidden) constants in $\bigO{\eps}$.  As long as $s \ll m$ (i.e., several orders of magnitude smaller than $m$), we do not expect it to affect the bounds.  Indeed, in Section~\ref{sec:experiments}, we look at examples with $s=2$ and $s=10$ and observe no dependence on block size.  In high-performance implementations, practical choices for $s$ depend on the application and hardware but typically remain small for a large number of unknowns.
\end{remark}

\begin{remark}
    Communication-avoiding Krylov subspace methods like $s$-step GMRES \cite{Hoe10} typically use a block Gram-Schmidt orthogonalization scheme and in each outer iteration generate a block vector of the form $\bmat{p_0(A)\vv & p_1(A)\vv & \cdots & p_{s-1}(A) \vv}$, where $A$ is a linear operator, $\vv$ is some starting vector, and $p_i$, $i \in \{0, \ldots, s-1\}$, are polynomials of degree $i$, respectively.  The methods considered in this manuscript can be used as block skeletons in $s$-step GMRES, but the analysis of its backward stability is more complicated than simply applying our results, due to the dual role that $s$ plays.  In the present manuscript, $s$ denotes a block partitioning of a \emph{fixed} matrix $\bXX$ and thus does not affect $\kappa(\bXX)$. Conversely, in $s$-step GMRES, $s$ determines not only the size of block vectors but also the conditioning of each block (and therefore the entire basis), as each is computed from powers of $A$.  Consequently the backward stability of $s$-step GMRES \emph{is} sensitive to the choice of $s$; for a complete analysis, see \cite{CarM24}, especially Figure~1 therein. In particular, note that employing \BCGSPIPIRO for orthogonalization in $s$-step GMRES may result in a limited backward error for certain examples, as demonstrated in Figures 8 and 9 of~\cite{CarM24}.
\end{remark}

%% file: sec_3_mp.tex
\section{Mixed-precision variants} \label{sec:mp}
It is possible to use multiple precisions in the implementations of \BCGSPIPRO and \BCGSPIPIRO without affecting the validity of general results from Section~\ref{sec:pip_variants}.  We provide pseudocode for two-precision versions of each as Algorithms~\ref{alg:PIP+MP} and Algorithm~\ref{alg:PIPI+MP}, respectively.  In the pseudocode, $\epslo$ denotes computing and storing a quantity in the low precision with the associated unit roundoff, and $\epshi$ likewise for high precision.  In particular, $\epslo \geq \epshi$.

One motivation for using multiple precisions is to attempt to eliminate the restriction on $\kappa(\bXX)$ for the stability bounds and thereby extend stability guarantees for higher condition numbers; see related work in, e.g., \cite{Okt24, OktC23}.  In a uniform working precision, both \BCGSPIPRO and \BCGSPIPIRO require $\bigO{\eps} \kappa^2(\bXX) \leq 1$, which practically translates into $\kappa(\bXX) \leq \bigO{10^{4}}$ or $\kappa(\bXX) \leq \bigO{10^{8}}$, for single or double precisions, respectively.  It is natural to consider whether using double the working precision in some parts of the algorithm might alleviate this restriction.  At the same time, we do not want to increase communication cost by transmitting high-precision data.  Therefore Algorithms~\ref{alg:PIP+MP}-\ref{alg:PIPI+MP} are formulated so that the data and solutions $\bQQ$ and $\RR$ are stored in low precision, while high precision is used for the local (i.e., on-node) computation of operations like $\vV^T \vV$, Cholesky factorization, and inverting Cholesky factors, with the motivation being that the Pythagorean step, based on Cholesky factorization, is ultimately responsible for the condition number restriction.  Note that the inner products in line \ref{line:BCGSPIP:innerprod} of \BCGSPIP and lines \ref{line:PIPI+:vSk} and \ref{line:PIPI+:vTk} of \BCGSPIPIRO are now split across two steps to handle different precisions.  However, we still regard these a single sync point, as the synchronization itself just involves the movement of memory, which can of course be handled in multiple precisions.

At the same time, doubling the precision implies doubling the computational cost and (potentially) the amount of data moved. This overhead is highly dependent on the problem size and it may be negligible in particular cases, such as latency-bound regimes\footnote{That is, where \emph{latency}, or the time it takes for memory to travel from one node to another across a network or between levels of cache, dominates the runtime of an algorithm.} and when both precisions are implemented in hardware.  When high precision computations are performed locally, such as in line~\ref{line:PIPI+:Uk} in Algorithm~\ref{alg:PIPI+}, the extra overhead may very well be insignificant.

\begin{algorithm}[htbp!]
	\caption{$[\bQQ, \RR] = \BCGSPIPMP(\bXX, \IOnoarg)$ \label{alg:PIPMP}}
	\begin{algorithmic}[1]
		\State{$[\vQ_1, R_{11}] = \IO{\vX_1}$} \SComment{compute and return in $\epslo$}
		\For{$k = 2,\ldots,p$}
                \State{$\RR_{1:k-1,k} = \bQQ_{k-1}^T\vX_k$} \SComment{compute and return in $\epslo$}
                \State{$P_k = \vX_k^T\vX_k$} \SComment{compute and return in $\epshi$}
		    \State{$R_{kk} = \chol\bigl( P_k - \RR_{1:k-1,k}^T \RR_{1:k-1,k} \bigr)$} \SComment{compute and return in $\epshi$; cast to $\epslo$ after line~\ref{line:BCGSPIPMP:Q}}
		    \State{$\vV_k = \vX_k - \bQQ_{k-1} \RR_{1:k-1,k}$} \SComment{compute and return in $\epslo$}
		    \State{$\vQ_k = \vV_k R_{kk}^\inv$} \SComment{compute each $s \times s$ block locally in $\epshi$; return in $\epslo$}\label{line:BCGSPIPMP:Q}
		\EndFor
		\State \Return{$\bQQ = [\vQ_1, \ldots, \vQ_p]$, $\RR = (R_{ij})$}
	\end{algorithmic}
\end{algorithm}

\begin{algorithm}[htbp!]
    \caption{$[\bQQ, \RR] = \BCGSPIPROMP(\bXX, \IOnoarg)$ \label{alg:PIP+MP}}
    \begin{algorithmic}[1]
        \State{$[\bUU, \SS] = \BCGSPIPMP(\bXX, \IOnoarg)$} \SComment{computed in mixed; returned in $\epslo$}
        \State{$[\bQQ, \TT] = \BCGSPIPMP(\bUU, \IOnoarg)$} \SComment{computed in mixed; returned in $\epslo$}
        \State{$\RR = \TT \SS;$} \SComment{compute and return in $\epslo$}
        \State \Return{$\bQQ = [\vQ_1, \ldots, \vQ_p]$, $\RR = (R_{ij})$}
    \end{algorithmic}
\end{algorithm}

\begin{algorithm}[htbp!]
    \caption{$[\bQQ, \RR] = \BCGSPIPIROMP(\bXX, \IOnoarg)$ \label{alg:PIPI+MP}}
    \begin{algorithmic}[1]
        \State{$[\vQ_1, R_{11}] = \IO{\vX_1}$} \SComment{compute and return in $\epslo$}
        \For{$k = 2, \ldots, p$}
            \State{$\SS_{1:k-1,k} = \bQQ_{k-1}^T \vX_k$} \SComment{compute and return in $\epslo$}\label{line:PIPI+MP:vSk}
            \State{$\Omega_k = \vX_k^T \vX_k$} \SComment{compute and return in $\epshi$}\label{line:PIPI+MP:Omveck}
            \State{$S_{kk} = \chol\bigl( \Omega_k - \SS_{1:k-1,k}^T \SS_{1:k-1,k} \bigr)$} \SComment{compute and return in $\epshi$}
            \State{$\vV_k = \vX_k - \bQQ_{k-1} \SS_{1:k-1,k}$} \SComment{compute and return in $\epslo$}\label{line:PIPI+MP:Vk}
            \State{$\vU_k = \vV_k S_{kk}^\inv$} \SComment{compute each $s \times s$ block locally in $\epshi$; return in $\epslo$}
            \State{$\TT_{1:k-1,k} = \bQQ_{k-1}^T \vU_k$} \SComment{compute and return in $\epslo$}
            \State{$P_k = \vU_k^T \vU_k$} \SComment{compute and return in $\epshi$}
            \State{$T_{kk} = \chol\bigl( P_k - \TT_{1:k-1,k}^T \TT_{1:k-1,k} \bigr)$} \SComment{compute and return in $\epshi$}
            \State{$\vW_k = \vU_k - \bQQ_{k-1} \TT_{1:k-1,k}$} \SComment{compute and return in $\epslo$}
            \State{$\vQ_k = \vW_k T_{kk}^\inv$} \SComment{compute each $s \times s$ block locally in $\epshi$; return in $\epslo$}
            \State{$\RR_{1:k-1,k} = \SS_{1:k-1,k} + \TT_{1:k-1,k} S_{kk}$} \SComment{compute and return in $\epslo$}
            \State{$R_{kk} = T_{kk} S_{kk}$} \SComment{compute in $\epshi$; return in $\epslo$}
        \EndFor
        \State \Return{$\bQQ = [\vQ_1, \ldots, \vQ_p]$, $\RR = (R_{ij})$}
    \end{algorithmic}
\end{algorithm}

Unfortunately, our intuitive proposals for mixed-precision variants do not achieve the desired stability for either \BCGSPIPROMP or \BCGSPIPIROMP, which we demonstrate in the following sections.

% =========================================================
\subsection{Two-precision \texttt{BCGS-PIP} and \texttt{BCGS-PIP+}} \label{sec:BCGSPIPROMP}

With Theorem~\ref{thm:PIP+}, we already have generalized bounds on the LOO that will also hold for \BCGSPIPROMP.  To determine what the constants $\deltaUS, \omegaU, \deltaQT, \omegaQ$ and $\deltaTS$ look like in the two-precision case, we need to obtain generalized bounds for \BCGSPIP that do not explicitly rely on a uniform precision but rather express bounds in terms of different constants with subscripts corresponding to the source of error.  The following two lemmas take care of this; they generalize \cite[Theorems~3.1 and 3.2]{CarLR21}, respectively. Proofs of Lemma \ref{lem:PIP:3.1} and Lemma \ref{lem:PIP:3.2} are available in Appendix \ref{app:lem:PIP:3.1} and \ref{app:lem:PIP:3.2}, respectively.

\begin{lemma} \label{lem:PIP:3.1}
    Let $\bXX \in \spR^{m \times ps}$ and $[\bQQbar, \RRbar] = \BCGSPIP(\bXX, \IOnoarg)$, for some $\IOnoarg$. Suppose $\xi \kappa^2(\bXX) \leq \frac{1}{2}$ and $\rho \kappa^2(\bXX) \leq \frac{1}{4}$, for constants $\xi, \rho \in (0,1)$.  Furthermore, assume that
    \begin{align}
        \RRbar^T \RRbar &= \bXX^T \bXX + \DeltaEE, \quad \norm{\DeltaEE} \leq \xi \norm{\bXX}^2, \mbox{ and} \label{eq:lem:PIP:3.1:cholres} \\
        \bQQbar \RRbar &= \bXX + \DeltabDD, \quad \norm{\DeltabDD} \leq \rho (\norm{\bXX} + \norm{\bQQbar} \norm{\RRbar} ). \label{eq:lem:PIP:3.1:res}
    \end{align}
    Then
    \begin{align}
        \norm{I - \bQQbar^T \bQQbar}
        & \leq \frac{\bigl(\xi + 11 \cdot \rho \bigr) \kappa^2(\bXX)}{1 - \xi \kappa^2(\bXX)}; \label{eq:lem:PIP:3.1:LOO} \\
        \norm{\bQQbar}
        & \leq 3; \mbox{ and} \label{eq:lem:PIP:3.1:normQ} \\
        \norm{\DeltabDD}
        & \leq 6 \cdot \rho \norm{\bXX}. \label{eq:lem:PIP:3.1:D}
    \end{align}
\end{lemma}

For the following, we assume that for each $k \in \{2, \ldots, p\}$, $\RRbar_{k-1}$ and $\bQQbar_{k-1}$ are computed by \BCGSPIP and that there exist $\xi_{k-1}, \rho_{k-1} \in (0,1)$ such that
\begin{equation} \label{eq:PIP:cholres}
    \begin{split}
        \RRbar_{k-1}^T \RRbar_{k-1} & = \bXX_{k-1}^T \bXX_{k-1} + \DeltaEE_{k-1}, \\
        \norm{\DeltaEE_{k-1}} & \leq \xi_{k-1} \norm{\bXX_{k-1}}^2;
    \end{split}
\end{equation}
and
\begin{equation} \label{eq:PIP:res}
    \begin{split}
        \bQQbar_{k-1} \RRbar_{k-1} & = \bXX_{k-1} + \DeltabDD_{k-1}, \\
        \norm{\DeltabDD_{k-1}} & \leq \rho_{k-1} (\norm{\bXX_{k-1}} + \norm{\bQQbar_{k-1}} \norm{\RRbar_{k-1}} ).
    \end{split}
\end{equation}
We can furthermore write the following for intermediate quantities computed by \BCGSPIP, where we omit the explicit dependence on $k$ for readability and each $\delta_{*} \in (0,1)$:
\begin{align}
    \RRbar_{1:k-1,k} & = \bQQbar^T_{k-1} \vX_k + \DeltavR_k,
    \quad \norm{\DeltavR_k} \leq \deltaQX \norm{\vX_k}; \label{eq:PIP:vRk} \\
    \Pbar_k & = \vX_k^T \vX_k + \DeltaP_k, 
    \quad \norm{\DeltaP_k} \leq \deltaXX \norm{\vX_k}^2; \label{eq:PIP:Pk} \\
    \Rbar_{kk}^T \Rbar_{kk} & = \Pbar_k - \RRbar^T_{1:k-1,k} \RRbar_{1:k-1,k} + \DeltaF_k + \DeltaC_k, \label{eq:PIP:Rkk} \\
    \norm{\DeltaF_k} & \leq \deltaRR \norm{\vX_k}^2, \quad \norm{\DeltaC_k} \leq \deltachol \norm{\vX_k}^2; \label{eq:PIP:chol} \\
    \vVbar_k & = \vX_k - \bQQbar_{k-1} \RRbar_{1:k-1,k} + \DeltavV_k,
    \quad \norm{\DeltavV_k} \leq \deltaQR \norm{\vX_k}; \mbox{ and} \label{eq:PIP:Vk} \\
    \vQbar_k \Rbar_k &= \vVbar_k + \DeltavG_k,
    \quad \norm{\DeltavG_k} \leq \deltaQ \norm{\vQbar_k} \norm{\Rbar_k}. \label{eq:PIP:Qk}
\end{align}
Throughout \eqref{eq:PIP:vRk}--\eqref{eq:PIP:Qk}, we have dropped quadratic error terms and applied\\Lemma~\ref{lem:PIP:3.1} to simplify constants (similarly to what we have done in \eqref{eq:PIPI+:vSk}--\eqref{eq:PIPI+:Tkk}; the contribution of $\norm{\bQQbar_{k-1}}$ is absorbed into the constant).  Furthermore, $\DeltaF_k$ denotes the floating-point error from the subtraction of the product $\RRbar^T_{1:k-1,k} \RRbar_{1:k-1,k}$ from $\Pbar_k$, and $\DeltaC_k$ denotes the error from the Cholesky factorization of that result; note that, similar to the argument in the proof of \cite[Theorem~3.2]{CarLR21}, $\Pbar_k - \RRbar^T_{1:k-1,k} \RRbar_{1:k-1,k}$ should be symmetric positive definite.

\begin{lemma} \label{lem:PIP:3.2}
    Let $\bXX \in \spR^{m \times ps}$ and fix $k \in \{2, \ldots, p\}$. Assume that \eqref{eq:PIP:cholres}--\eqref{eq:PIP:Qk} are satisfied.  Then the following hold:
    \begin{equation} \label{eq:lem:PIP:3.2:cholres}
        \RRbar_k^T \RRbar_k = \bXX_k^T \bXX_k + \DeltaEE_k,
        \quad \norm{\DeltaEE_k} \leq \xi_k \norm{\bXX_k}^2,
    \end{equation}
    and
    \begin{equation} \label{eq:lem:PIP:3.2:res}
        \bQQbar_k \RRbar_k = \bXX_k + \DeltabDD_k,
        \quad \norm{\DeltabDD_k} \leq \rho_k (\norm{\bXX_k} + \norm{\bQQbar_k} \norm{\RRbar_k}),
    \end{equation}
    where
    \[
    \xi_k = \xi_{k-1} + 12 \cdot \rho_{k-1} + 2 \sqrt{2} \cdot \deltaQX + \deltaXX+ \deltachol + \deltaRR
    \]
    and $\rho_k = 6 \cdot \rho_{k-1} + \deltaQR + \deltaQ$.
\end{lemma}

On their own, Lemmas~\ref{lem:PIP:3.1} and \ref{lem:PIP:3.2} do not complete the analysis; they describe the relationship between bounds from one iteration to the next without specifying the precision.  Set
\begin{equation} \label{eq:PIP:rho_xi}
    \rhomax := \max_{k \in \{1, \ldots, p-1\}} \rho_k \quad\mbox{and}\quad \ximax := \max_{k \in \{1, \ldots, p-1\}} \xi_k. 
\end{equation}
In the next theorem, we combine these lemmas to obtain bounds on \BCGSPIPMP. A proof of the theorem is available in Appendix \ref{app:thm:PIP-MP}.

\begin{theorem} \label{thm:PIP-MP} 
    Let $\bXX \in \spR^{m \times ps}$ such that
    \begin{equation} \label{eq:thm:PIP-MP:assump}
        (\rhomax + \rho_p) \kappa^2(\bXX) \leq \frac{1}{4} \quad\mbox{and}\quad (\ximax + \xi_p) \kappa^2(\bXX) \leq \frac{1}{2},
    \end{equation}
    for $\rho_p$ and $\xi_p$ defined as in Lemma~\ref{lem:PIP:3.2}.  Suppose $[\bQQbar, \RRbar] = \BCGSPIPMP(\bXX, \IOnoarg)$, where for all $\vX \in \spR^{m \times s}$ with $\kappa(\vX) \leq \kappa(\bXX)$, $[\vQbar, \Rbar] = \IO{\vX}$ satisfy
    \begin{align}
        \Rbar^T \Rbar & = \vX^T \vX + \DeltaE,
        \quad \norm{\DeltaE} \leq \bigO{\epslo} \norm{\vX}^2 \mbox{ and} \label{eq:thm:PIP-MP:IO:cholres} \\ 
        \vQbar \Rbar & = \vX + \DeltavD, \quad \norm{\DeltavD} \leq \bigO{\epslo} \norm{\vX}. \label{eq:thm:PIP-MP:IO:res}
    \end{align}
    Then for all $k \in \{1, \ldots, p\}$,
    \begin{align}
        \RRbar_k^T \RRbar_k
        & = \bXX_k^T \bXX_k + \DeltaEE_k,
        \quad \norm{\DeltaEE_k} \leq \bigO{\epslo} \norm{\bXX_k}^2; \label{eq:thm:PIP-MP:cholres} \\
        \bQQbar_k \RRbar_k 
        & = \bXX_k + \bDD_k, \quad \norm{\bDD_k} \leq \bigO{\epslo} \norm{\bXX_k}; \mbox{ and} \label{eq:thm:PIP-MP:res} \\
        \norm{I - \bQQbar_k^T \bQQbar_k}
        & \leq \bigO{\epslo} \kappa^2(\bXX_k). \label{eq:thm:PIP-MP:LOO}
    \end{align}
\end{theorem}

Applying Theorem~\ref{thm:PIP-MP} twice to the dual-precision \BCGSPIPROMP shows that
\begin{equation} \label{eq:PIP+MP:constants}
    \deltaUS, \omegaU, \deltaQT, \omegaQ, \deltaTS = \bigO{\epslo},
\end{equation}
where the constants stem from \eqref{eq:PIP+:LOO:U}--\eqref{eq:PIP+:R=TS}.  The following corollary is a direct consequence of \eqref{eq:PIP+MP:constants} and Theorem~\ref{thm:PIP+} and demonstrates that \BCGSPIPROMP is dominated by $\bigO{\epslo}$ error.

\begin{corollary} \label{cor:PIP+MP}
    Let $\bXX \in \spR^{m \times ps}$ such that $\delta_{\max} \kappa^2(\bXX) \leq \frac{1}{4}$,
    where $\delta_{\max} := \max\{ \deltaUS, \omegaU, \deltaQT, \omegaQ \}$.  Suppose $[\bQQbar, \RRbar] = \BCGSPIPMP(\bXX, \IOnoarg)$, where for all $\vX \in \spR^{m \times s}$ with $\kappa(\vX) \leq \kappa(\bXX)$, $[\vQbar, \Rbar] = \IO{\vX}$ satisfy
    \begin{align*}
        \Rbar^T \Rbar & = \vX^T \vX + \DeltaE,
        \quad \norm{\DeltaE} \leq \bigO{\epslo} \norm{\vX}^2 \mbox{ and} \\ 
        \vQbar \Rbar & = \vX + \DeltavD, \quad \norm{\DeltavD} \leq \bigO{\epslo} \norm{\vX}.
    \end{align*}
    Then for all $k \in \{1, \ldots, p\}$,
    \begin{align*}
        \RRbar_k^T \RRbar_k
        & = \bXX_k^T \bXX_k + \DeltaEE_k,
        \quad \norm{\DeltaEE_k} \leq \bigO{\epslo} \norm{\bXX_k}^2; \\
        \bQQbar_k \RRbar_k 
        & = \bXX_k + \bDD_k, \quad \norm{\bDD_k} \leq \bigO{\epslo} \norm{\bXX_k}; \mbox{ and} \\
        \norm{I - \bQQbar_k^T \bQQbar_k}
        & \leq \bigO{\epslo}.
    \end{align*}
\end{corollary}

% =========================================================
\subsection{Two-precision \texttt{BCGS-PIPI+}} \label{sec:BCGSPIPIROMP}

By standard rounding-error analysis we can show that for \eqref{eq:PIPI+:vSk}--\eqref{eq:PIPI+:Tkk} in \BCGSPIPIROMP,
\[
    \deltaQX, \deltaQS, \deltaQU, \deltaQT = \bigO{\epslo}
\]
and
\[
    \deltaXX, \deltaSS, \deltacholone, \deltaU, \deltaUU, \deltaTT, \deltacholtwo, \deltaQ = \bigO{\epshi}.
\]
The following corollary follows directly from these constants and the development in Section~\ref{sec:bcgs_pipi_ro}, in a similar manner as Corollaries~\ref{cor:PIPI+:LOO}, \ref{cor:PIPI+:res}, and \ref{cor:PIPI+:cholres}.  Consequently, we must conclude that \BCGSPIPIROMP is also dominated by $\bigO{\epslo}$ error.

\begin{corollary} \label{cor:PIPI+MP}
    Assume that $\bigO{\epslo} \kappa^2(\bXX) \leq \frac{1}{2}$ and that for all $\vX \in \spR^{m \times s}$ with $\kappa(\vX) \leq \kappa(\bXX)$, $[\vQbar, \Rbar] = \IO{\vX}$ satisfy
    \begin{align*}
        \Rbar^T \Rbar
        & = \vX^T \vX + \DeltaE, \quad \norm{\DeltaE} \leq \bigO{\epslo} \norm{\vX}^2 \\ 
        \vQbar \Rbar
        & = \vX + \DeltavD, \quad \norm{\DeltavD} \leq \bigO{\epslo} \norm{\vX} \mbox{ and} \\
        \norm{I - \vQbar^T \vQbar}
        & \leq \frac{\bigO{\epslo}}{1 - \bigO{\epslo} \kappa^2(\vX)}. 
    \end{align*}
    Then for $[\bQQbar, \RR] = \BCGSPIPIROMP(\bXX, \IOnoarg)$, the following hold for all $k \in \{1, \ldots, p\}$:
    \begin{align*}
        \RRbar_k^T \RRbar_k
        & = \bXX_k^T \bXX_k + \DeltaEE_k, \quad \norm{\DeltaEE_k} \leq \bigO{\epslo} \norm{\bXX_k}^2; \\
        \bQQbar_k \RRbar_k
        & = \bXX_k + \DeltabDD_k, \quad \norm{\DeltabDD_k} \leq \bigO{\epslo} \norm{\bXX_k} \\
        \norm{I - \bQQbar_k^T \bQQbar_k}
        & \leq \frac{\bigO{\epslo}}{1 - \bigO{\epslo} \kappa^2(\bXX_k)} \leq \bigO{\epslo}.
    \end{align*}
\end{corollary}

\begin{remark}
    Our analytical framework is more general than for an arbitrary uniform precision or even two precisions, which we have focused on.  Indeed, it is not hard to see that regardless of the number of precisions used in these algorithms, the lowest precision will dominate the LOO and residual bounds.  At the same time, numerical experiments in the next section demonstrate that our bounds are rather pessimistic for practical scenarios.  More nuanced bounds remain a topic for future work.
\end{remark}

%% file: sec_4_experiments.tex
\section{Numerical experiments} \label{sec:experiments}
\subsection{\texttt{BlockStab}: a workflow for BGS comparisons}
We make use of an updated version of the \texttt{BlockStab} code suite \cite{LunOCetal24}\footnote{\url{https://github.com/katlund/BlockStab/releases/tag/v2.1.2024}}.  In this version, multiprecision implementations have been added for both the Advanpix Multiprecision Computing Toolbox\footnote{Version 4.8.3.14460. \url{https://www.advanpix.com/}} and the Symbolic Math Toolbox\footnote{Version depends on MATLAB version. \url{https://mathworks.com/products/symbolic.html}}, along with additional low-sync versions of various block Gram-Schmidt routines, which are analyzed further in \cite{CarLMetal24b}.  The procedure for setting up algorithm comparisons is now streamlined to avoid redundant runs and to allow for different choices of Cholesky factorization implementations.  Finally, a TeX report is auto-generated with a timestamp, which facilitates sharing experiments with collaborators.

\subsection{Comparisons among newly proposed variants}
We perform several numerical experiments to study the stability of \BCGSPIP, \BCGSPIPRO, and\\\BCGSPIPIRO, using four classes of matrices available in \texttt{BlockStab}, namely \default, \glued, \monomial, and \piled. Each matrix class is created using dimensional inputs $m,p,s$, where $m$ denotes the number of rows, $p$ denotes the number of block vectors, and $s$ denotes the number of columns in each block vector. Descriptions of the matrix classes are as follows:
\begin{itemize}
    \item \default: built as $\bXX_t = \bUU \Sigma_t \bVV^T$, where $\bUU \in \spR^{m \times ps}$ is orthonormal, $\bVV \in \spR^{ps\times ps}$ is unitary, and $\Sigma_t \in \spR^{ps \times ps}$ is diagonal with entries drawn from the logarithmic interval $10^{[-t,0]}$.
    \item \glued: first introduced in \cite{SmoBL06} and constructed to cause classical Gram-Schmidt to break down.  The matrix is initialized with a \default matrix $\bXX_t$; then each $m \times s$ block vector is multiplied by $\Sigma_r \Vtil$, where $\Sigma_r, \Vtil \in \spR^{s \times s}$, and $\Vtil$ is unitary.
    \item \monomial: consists of $r$ block vectors $\vX_k = \bmat{\vv_k & A\vv_k & \cdots & A^{t-1} \vv_k}$, $k \in \{1, \ldots, r\}$, where each $\vv_k$ is randomly generated from the uniform distribution and normalized, and $A$ is an $m \times m$ diagonal operator having evenly distributed eigenvalues in $(0.1,10)$.   (Note that $r$ is not necessarily equal to $p$, and likewise $t \neq s$; however $rt = ps$.  Varying $r$ and $t$ allows for generating matrices with different condition numbers.)
    \item \piled: formed as $\bXX = \bmat{\vX_1 & \vX_2 & \cdots & \vX_p}$, where $\vX_1$ is a \default matrix with a small condition number and for $k \in \{2,\ldots,p\}$, $\vX_k = \vX_{k-1} + \vZ_k$, where each $\vZ_k$ is also a \default matrix with the same condition number for all $k$. Toggling the condition numbers of $\vX_1$ and $\{\vZ_k\}_{k=2}^p$ controls the overall conditioning of the test matrix.
\end{itemize}
We set $m = 100$, $p = 10$, and $s=2$ for \glued and \default matrices; $m = 2000$, $p = 120$, $s = 10$ for \monomial matrices; and $m = 100$, $p = 10$, $s = 5$ for \piled matrices.

To illustrate the numerical behavior of the algorithms from Sections~\ref{sec:pip_variants} and \ref{sec:mp}, we plot the LOO~\eqref{eq:loo} and relative Cholesky residual (i.e., \eqref{eq:chol_res} divided by $\norm{\bXX}^2$) of each algorithm versus the condition number of the matrix; we refer to these plots as $\kappa$-plots, as they are relative to the changing condition number $\kappa(\bXX)$. To observe the effects of the choice of \IOnoarg, we use \HouseQR and \CholQR, where a variant of Cholesky factorization is used to bypass MATLAB's \texttt{chol} protocol for halting the computation when a matrix loses numerical positive definiteness. One can regard \HouseQR as a placeholder for \TSQR, as their numerical behavior is similar, even though the communication properties would differ in practical distributed computing settings.

Double precision ($\eps = 2^{-53} \approx 10^{-16}$) is used for uniform-precision methods.  Advanpix is used to simulate quadruple precision ($\epshi = 2^{-113} \approx 10^{-32}$) in multiprecision algorithms, while the low precision is set to double ($\epslo = \eps$).\footnote{
Users who don't have access to additional toolboxes can still reproduce the trends in our multiprecision results by replacing \texttt{"mp\_pair":["double", "quad"]} in \url{https://github.com/katlund/BlockStab/blob/master/tests/configs/bcgs_pip_reortho.json} with \texttt{"mp\_pair":["single", "double"]}.
}

All numerical tests are run in MATLAB 2022a. Every test is run on a Lenovo ThinkPad E15 Gen 2 with 8GB memory and AMD Ryzen 5 4500U CPU with Radeon Graphics. The CPU has 6 cores with 384KiB L1 cache, and 3MiB L2 cache at a clockrate of 1 GHz, as well as 8MiB of shared L3 cache.  The script\\\texttt{test\_bcgs\_pip\_reortho.m} can be used for regenerating all plots in this section.

Figure~\ref{fig:pip_vs_reorth} illustrates the stability of reorthogonalized variants compared to\\\BCGSPIP in uniform precision. We see that \BCGSPIP follows a $\bigO{\eps} \kappa^2(\bXX)$ LOO trend until $\kappa(\bXX) \approx \frac{1}{\sqrt{\eps}} \approx 10^8$, as expected.  Meanwhile the reorthogonalized variants reach nearly $10^{-16}$ LOO, regardless of the choice of \IOnoarg, until $\kappa(\bXX) \approx 10^8$. After this point, we observe breakdowns and an increasing loss of orthogonality for all methods.  Missing points for large condition numbers are due to \texttt{NaN} being computed during the Cholesky factorization, resulting from operations like $\frac{\texttt{Inf}}{\texttt{Inf}}$ or $\frac{\texttt{Inf}}{0}$.

\begin{figure}[htbp!]
	\begin{center}
	    \begin{tabular}{cc}
	          \resizebox{.325\textwidth}{!}{\includegraphics[trim={0 0 219pt 0},clip]{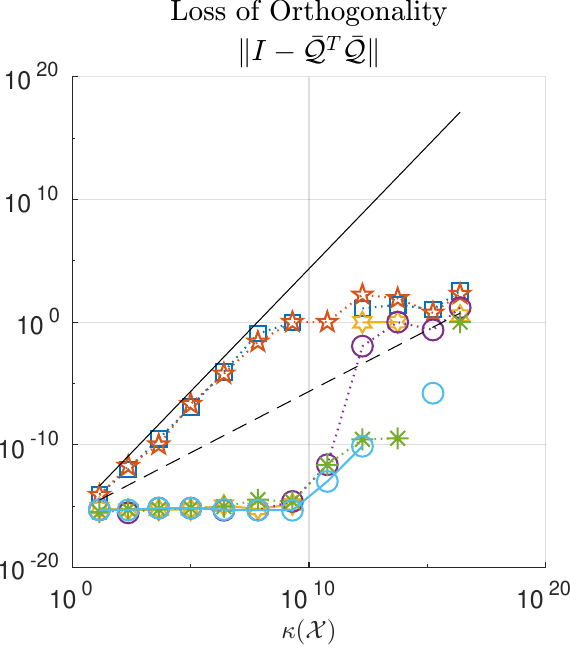}} &
	         \resizebox{.605\textwidth}{!}{\includegraphics{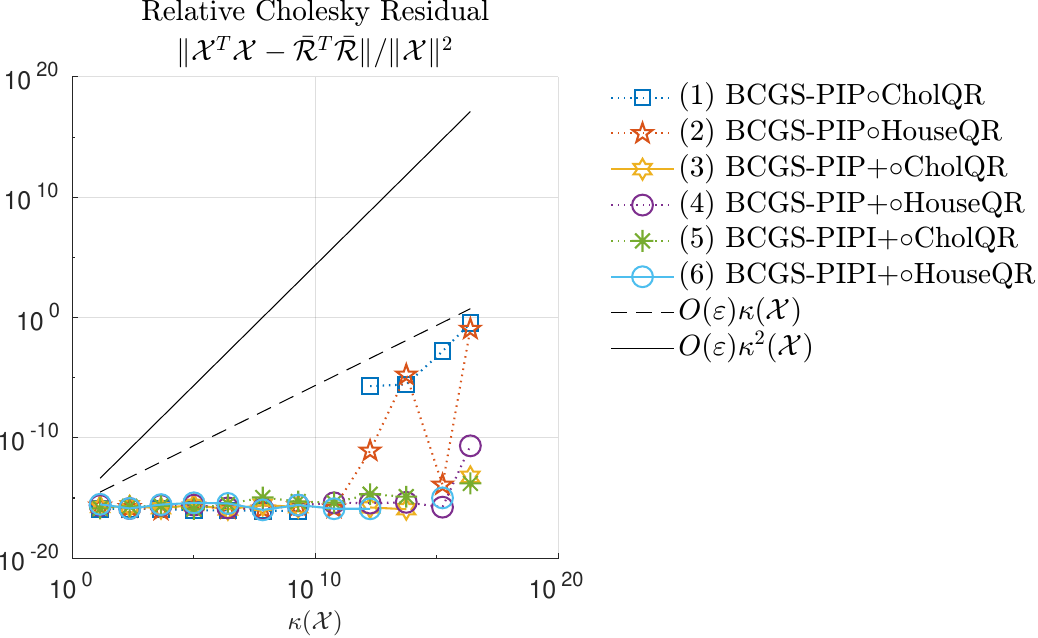}}
	    \end{tabular}
	\end{center}
	\caption{$\kappa$-plots for \glued matrices.\label{fig:pip_vs_reorth}}
\end{figure}

The effect of the choice of \IOnoarg for uniform-precision methods can be seen in Figure~\ref{fig:reorth}.  As there is no assumption on the LOO of the \IOnoarg in Corollary~\ref{cor:PIP+:LOO}, we observe that \CholQR works well for \BCGSPIPRO. On the other hand, Theorem~\ref{thm:PIPI+:LOO} places a LOO restriction on the \IOnoarg for \BCGSPIPIRO.  Indeed, \CholQR does not satisfy the assumption, so the proven LOO and residual bounds are not guaranteed to hold, and we can see $\BCGSPIPIRO\circ\CholQR$ failing to reach around $10^{-16}$ LOO even for small condition numbers.  Again, both \BCGSPIPRO and \BCGSPIPIRO exhibit an increasing LOO and residual after $\kappa(\bXX) > 10^8$.

\begin{figure}[htbp!]
	\begin{center}
	    \begin{tabular}{cc}
	          \resizebox{.33\textwidth}{!}{\includegraphics[trim={0 0 207pt 0},clip]{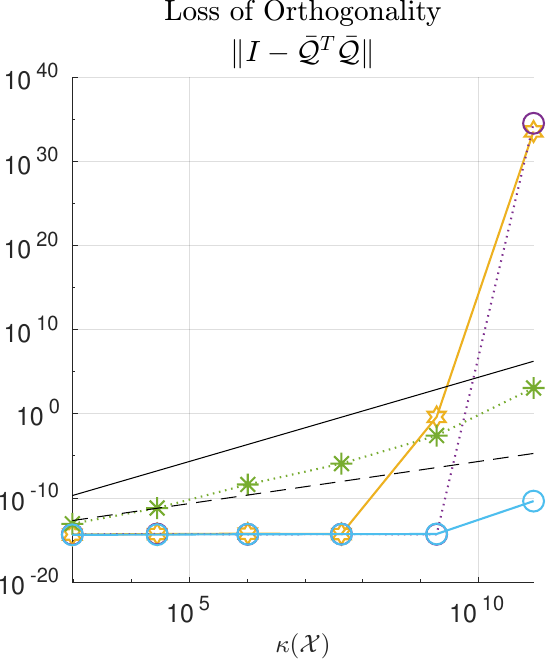}} &
	         \resizebox{.605\textwidth}{!}{\includegraphics{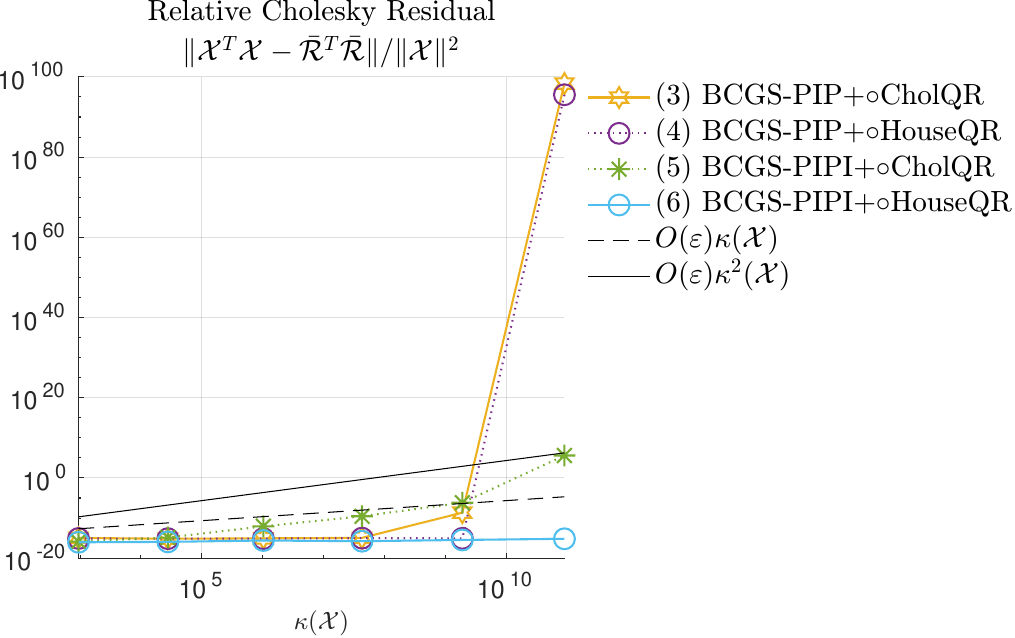}}
	    \end{tabular}
	\end{center}
	\caption{$\kappa$-plots for \monomial matrices.\label{fig:reorth}}
\end{figure}

Figure~\ref{fig:reorth_mp_default} compares two-precision with uniform-precision algorithms.  For \default matrices with or without mixed precision, reorthogonalization keeps the LOO near $10^{-16}$ as long as $\kappa(\bXX) \leq 10^8$ and below $\bigO{\eps} \kappa(\bXX)$ otherwise.  Although \BCGSPIPIROMP appears to behave well even for high condition numbers, the \default matrices are known to be ``easy" and may not capture all potential numerical behaviors.

\begin{figure}[htbp!]
	\begin{center}
	    \begin{tabular}{cc}
	         \resizebox{.3\textwidth}{!}{\includegraphics[trim={0 0 230pt 0},clip]{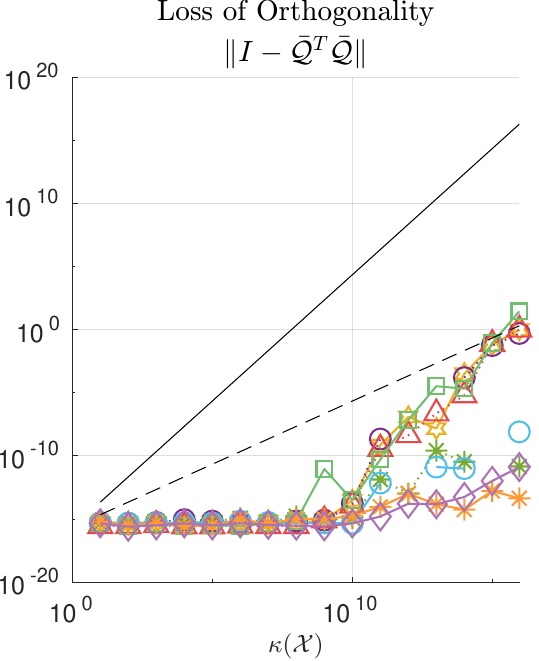}}&
	    \resizebox{.605\textwidth}{!}{\includegraphics{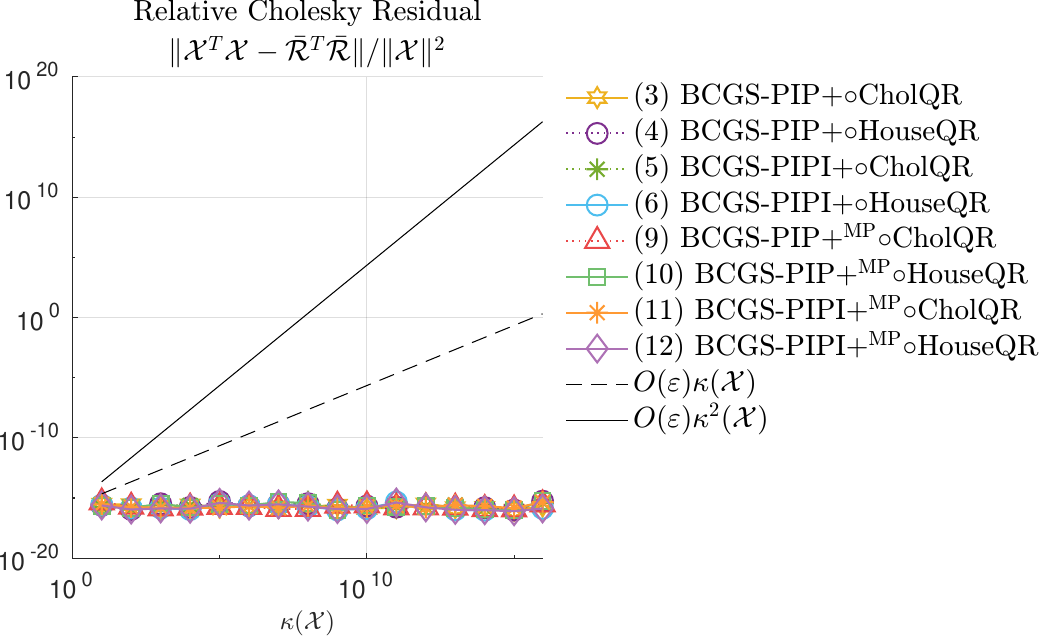}}
	    \end{tabular}
	\end{center}
	\caption{$\kappa$-plots for \default matrices. \label{fig:reorth_mp_default}}
\end{figure}

The \glued matrices are more challenging and reveal worse behavior in Figure~\ref{fig:reorth_mp_glued}. After $\kappa(\bXX) \approx 10^8$, \BCGSPIPRO has stability problems and the relative Cholesky residuals of $\BCGSPIPIRO\circ\CholQR$ and $\BCGSPIPRO\circ\HouseQR$ become \texttt{NaN}. Mixed precision overcomes this problem for both methods.  The LOO of \BCGSPIPIROMP remains below $\bigO{10^{-16}} \kappa(\bXX)$ whereas \BCGSPIPROMP begins to exceed this bound.  Notably, the behavior between $\BCGSPIPIROMP\circ\CholQR$ and $\BCGSPIPROMP\circ\HouseQR$ is very similar, despite the lack of guarantees for $\CholQR$.

\begin{figure}[htbp!]
	\begin{center}
	    \begin{tabular}{cc}
	          \resizebox{.307\textwidth}{!}{\includegraphics[trim={0 0 235pt 0},clip]{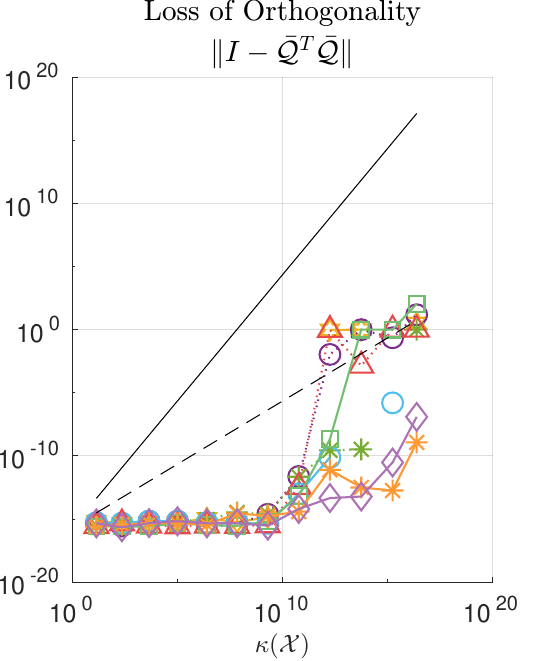}} &
	         \resizebox{.605\textwidth}{!}{\includegraphics{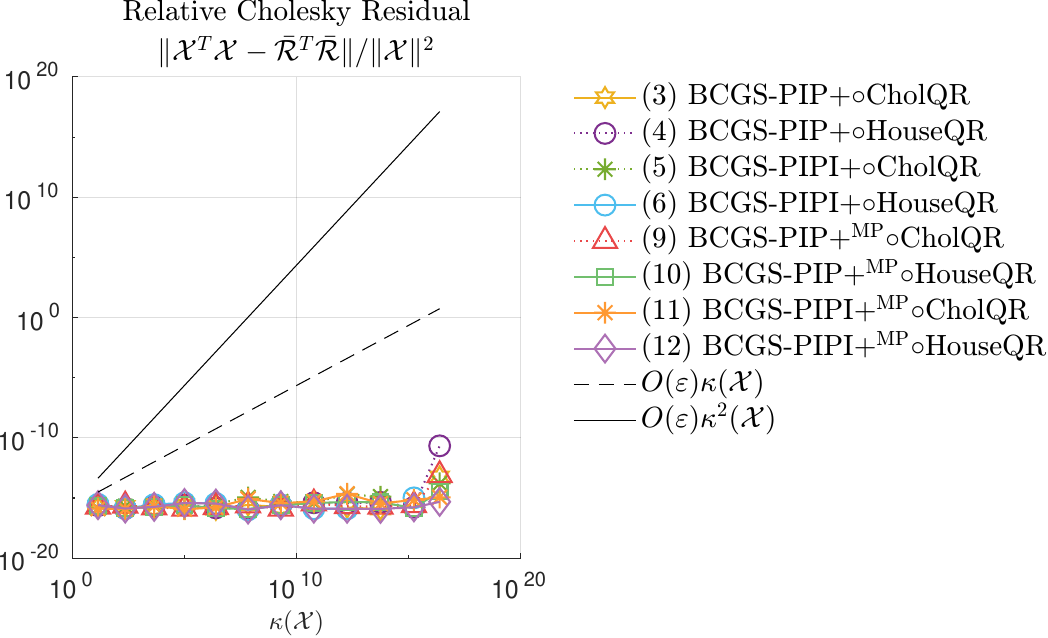}}
	    \end{tabular}
	\end{center}
	\caption{$\kappa$-plots for \glued matrices.\label{fig:reorth_mp_glued}}
\end{figure}

Figures~\ref{fig:reorth_mp_default} and \ref{fig:reorth_mp_glued} might trick the reader into concluding that $\BCGSPIPIROMP\circ\CholQR$ is rather reliable, even without theoretical bounds. The \piled matrices in Figure~\ref{fig:mp_piled} should dispel this notion.  Neither $\BCGSPIPIRO\circ\CholQR$ nor $\BCGSPIPIROMP\circ\CholQR$ can attain $10^{-16}$ LOO for even the smallest condition numbers.  And even $\BCGSPIPIROMP\circ\HouseQR$ finally manages to lose all orthogonality for $\kappa(\bXX) < 10^{16}$, which is still numerically nonsingular in double precision. Moreover, we observe rather erratic behavior in the LOO of both variants of $\BCGSPIPROMP$ once $\kappa(\bXX) \geq 10^{8}$, which emphasizes the lack of predictability outside of the bounds proven in Section~\ref{sec:pip_variants}.
\begin{figure}[htbp!]
	\begin{center}
	    \begin{tabular}{cc}
	          \resizebox{.295\textwidth}{!}{\includegraphics[trim={0 0 235pt 0},clip]{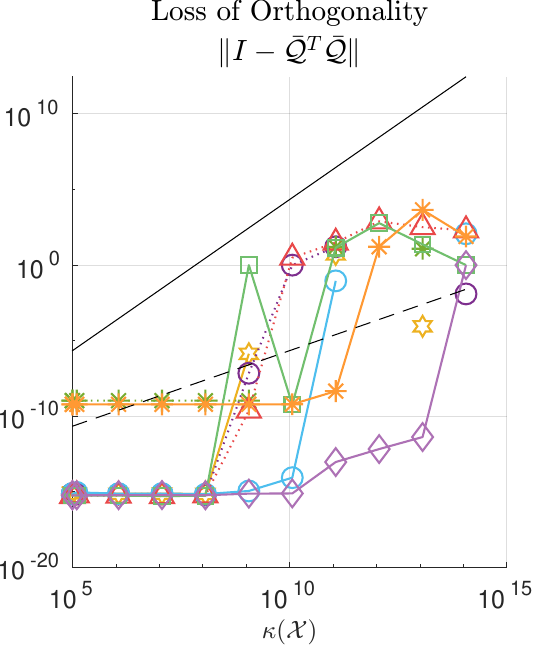}} &
	         \resizebox{.605\textwidth}{!}{\includegraphics{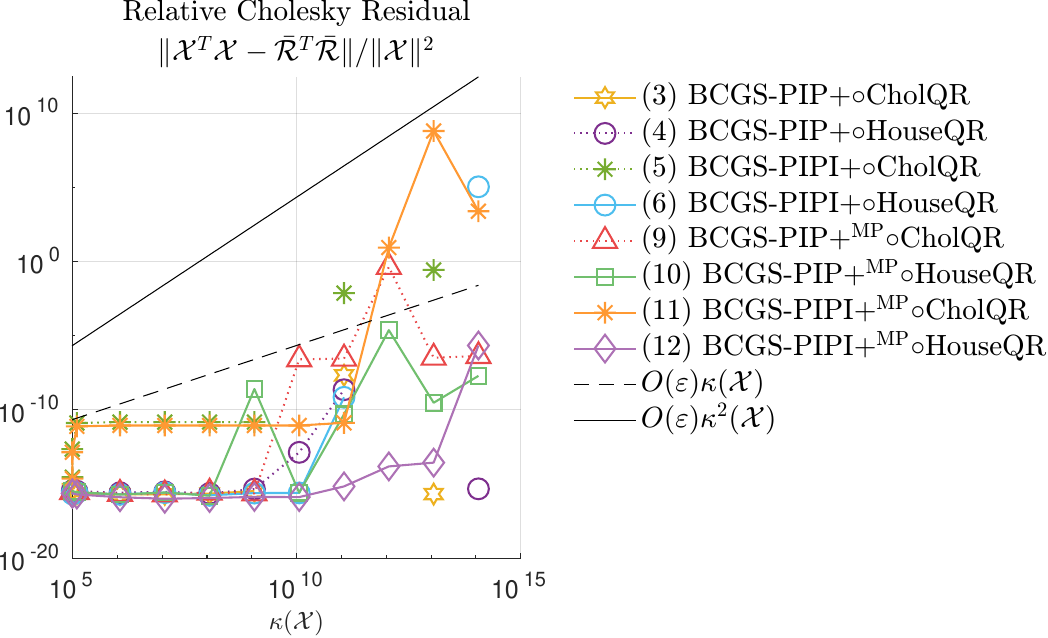}}
	    \end{tabular}
	\end{center}
	\caption{$\kappa$-plots for \piled matrices.\label{fig:mp_piled}}
\end{figure}

%% file: sec_5_conclusions.tex
\section{Conclusions} \label{sec:conclusions}
Reorthogonalization is a simple technique for regaining stability in a Gram-Schmidt procedure.  We have introduced and examined two reorthogonalized variants of \BCGSPIP, \BCGSPIPRO and \BCGSPIPIRO, and demonstrated that both can achieve $\bigO{\eps}$ loss of orthogonality under transparent conditions on their intraorthogonalization routines and on the condition number of $\bXX$, namely that $\bigO{\eps} \kappa^2(\bXX) \leq 1/2$.  We have carried out the analysis in a general enough fashion so that results can be easily extended to multiprecision paradigms, and we have proposed two-precision variants \BCGSPIPROMP and \BCGSPIPIROMP.  Numerical experiments verify our findings and demonstrate that despite the lack of theoretical bounds, \BCGSPIPIROMP behaves well for several classes of test matrices and nearly overcomes the restriction on $\kappa(\bXX)$.

At the same time, the restriction on $\kappa(\bXX)$ may not be so problematic when \BCGSPIPRO or \BCGSPIPIRO forms the backbone of a block Arnoldi or GMRES algorithm and can be restarted; see, e.g., \cite{Lun23, XuAD23, YamHBetal24}.  In fact, with the recent modular framework developed for the backward stability analysis of GMRES \cite{ButHMetal24}, determining reliable, adaptive restarting heuristics should be quite straightforward.  In such scenarios, one usually also has access to preconditioning, which can a priori reduce the conditioning of the basis to be orthogonalized and further improve overall stability.  As \BCGSPIPRO or \BCGSPIPIRO both require the same number of sync points as \BCGS, but with provably better loss of orthogonality, they are promising, stable algorithms for a wide variety of applications.

%% file: acknowledgements.tex
The second author would like to thank the Computational Methods in Systems and Control Theory group at the Max Planck Institute for Dynamics of Complex Technical Systems for funding the fourth author's visit in March 2023. The fourth author would like to thank the Chemnitz University of Technology for funding the second author's visit in July 2023.

%% file: funding.tex
The first, third, and fourth authors are supported by the European Union (ERC, inEXASCALE, 101075632). Views and opinions expressed are those of the authors only and do not necessarily reflect those of the European Union or the European Research Council. Neither the European Union nor the granting authority can be held responsible for them. The first and the fourth authors acknowledge support from the Charles University GAUK project No. 202722 and the Exascale Computing Project (17-SC-20-SC), a collaborative effort of the U.S.~Department of Energy Office of Science and the National Nuclear Security Administration. The first author additionally acknowledges support from the Charles University Research Centre program No.~UNCE/24/SCI/005.

%% file: appendix.tex
\section{Proofs of the theorems in Section \ref{sec:mp} }
\subsection{Proof of Lemma \ref{lem:PIP:3.1}}\label{app:lem:PIP:3.1}
\begin{proof}
    Assumption \eqref{eq:lem:PIP:3.1:cholres} implies
    \begin{equation} \label{eq:lem:PIP:3.1:normR}
        \norm{\RRbar}^2 \leq (1 + \xi) \norm{\bXX}^2,
    \end{equation}
    and together with $\xi \kappa^2(\bXX) \leq \frac{1}{2}$ and the perturbation theory of singular values~\cite[Corollary~8.6.2]{GolV13}, we can derive a bound on $\norm{\RRbar^\inv}^2$:
    \begin{equation} \label{eq:lem:PIP:3.1:normRinv}
        \norm{\RRbar^\inv}^2 = \dfrac{1}{\sigmin^2(\RRbar)} \leq \dfrac{1}{\sigmin^2(\bXX) (1 - \xi \kappa^2(\bXX) )}.
    \end{equation}
    Multiplying \eqref{eq:lem:PIP:3.1:res} on the left by its transpose, then on the left by $\RRbar^\tinv$ and on the right by $\RRbar^\inv$ and substituting \eqref{eq:lem:PIP:3.1:cholres} leads to
    \begin{equation} \label{eq:lem:PIP:3.1:QTQ}
        \bQQbar^T \bQQbar = I - \RRbar^\tinv (\DeltaEE - \bXX^T \DeltabDD - \DeltabDD^T \bXX) \RRbar^\inv,
    \end{equation}
    where we have dropped the quadratic term.  Taking the norm and applying bounds \eqref{eq:lem:PIP:3.1:cholres}, \eqref{eq:lem:PIP:3.1:res}, \eqref{eq:lem:PIP:3.1:normR}, and \eqref{eq:lem:PIP:3.1:normRinv} to \eqref{eq:lem:PIP:3.1:QTQ} leads to a quadratic inequality:
    \begin{equation} \label{eq:lem:PIP:3.1:normQ^2}
        \begin{split}
            \norm{\bQQbar}^2
            & \leq 1 + \norm{\RRbar^\inv}^2 \left( \xi \norm{\bXX}^2
            + 2 \rho  \norm{\bXX} (\norm{\bXX} + \norm{\bQQbar} \norm{\RRbar} )\right) \\
            & \leq 1 + \frac{ (\xi + 2 \rho) \norm{\bXX}^2
            + 2 \rho \sqrt{1 + \xi} \norm{\bXX}^2 \norm{\bQQbar}}{\sigmin^2(\bXX) (1 - \xi \kappa^2(\bXX) )} \\
            & \leq \frac{1 + 2 \rho \kappa^2(\bXX)}{1 - \xi \kappa^2(\bXX)}
            +  \frac{2\rho \sqrt{1 + \xi} \kappa^2(\bXX)}{1 - \xi \kappa^2(\bXX)} \norm{\bQQbar}.
        \end{split}
    \end{equation}
    Applying the assumptions $\xi \kappa^2(\bXX) \leq \frac{1}{2}$ and $\rho \kappa^2(\bXX) \leq \frac{1}{4}$, we see that
    \[
        \frac{1 + 2 \rho \kappa^2(\bXX)}{1 - \xi \kappa^2(\bXX)} \leq 3
        \mbox{ and }
        \frac{2 \rho \kappa^2(\bXX)}{1 - \xi \kappa^2(\bXX)} \leq 1,
    \]
    so \eqref{eq:lem:PIP:3.1:normQ^2} simplifies to
    \begin{equation} \label{eq:lem:PIP:3.1:normQ:quad}
        \norm{\bQQbar}^2 - \sqrt{1 + \xi} \norm{\bQQbar} - 3 \leq 0,
    \end{equation}
    which can be easily solved with the quadratic formula to reveal
    \begin{equation*}
        \norm{\bQQbar}
        \leq \frac{1}{2} \bigl(\sqrt{1 + \xi} + \sqrt{13 + \xi} \bigr)
        \leq \frac{1}{2} \bigl(\sqrt{2} + \sqrt{14} \bigr)
        \leq 3,
    \end{equation*}
    thus proving \eqref{eq:lem:PIP:3.1:normQ}.  Revisiting \eqref{eq:lem:PIP:3.1:QTQ}, we take the norm of $\norm{I - \bQQbar^T \bQQbar}$ and apply \eqref{eq:lem:PIP:3.1:normQ} along with bounds \eqref{eq:lem:PIP:3.1:cholres}, \eqref{eq:lem:PIP:3.1:res}, \eqref{eq:lem:PIP:3.1:normR}, \eqref{eq:lem:PIP:3.1:normRinv} again to arrive at \eqref{eq:lem:PIP:3.1:LOO}.  Finally, \eqref{eq:lem:PIP:3.1:D} follows from \eqref{eq:lem:PIP:3.1:normR} and \eqref{eq:lem:PIP:3.1:normQ}.
\end{proof}

\subsection{Proof of Lemma \ref{lem:PIP:3.2}}\label{app:lem:PIP:3.2}
\begin{proof}
    Writing $\RRbar_k$ as 
    \begin{equation*}
        \RRbar_k = \bmat{
            \RRbar_{k-1} & \RRbar_{1:k-1,k}\\
            0 & \Rbar_k
        }
    \end{equation*}
    implies
    \begin{equation} \label{eq:lem:PIP:3.2:blockR}
        \RRbar^T_k \RRbar_k = \bmat{
            \RRbar_{k-1}^T \RRbar_{k-1}      & \RRbar_{k-1}^T \RRbar_{1:k-1,k} \\
            \RRbar^T_{1:k-1,k} \RRbar_{k-1}  & \RRbar^T_{1:k-1,k} \RRbar_{1:k-1,k} + \Rbar^T_k \Rbar_k
        }.
    \end{equation}
    The upper diagonal block of \eqref{eq:lem:PIP:3.2:blockR} can be handled directly from \eqref{eq:PIP:cholres}. Lemma~\ref{lem:PIP:3.1} holds for $k-1$ via the assumptions \eqref{eq:PIP:cholres} and \eqref{eq:PIP:res}. Then for the off-diagonals, from \eqref{eq:lem:PIP:3.1:normQ} and \eqref{eq:PIP:vRk}, we can write
    \begin{equation} \label{eq:lem:PIP:normR-general}
        \norm{\RRbar_{1:k-1,k}}
        \leq \norm{\bQQbar_{k-1}} \norm{\vX_k} + \norm{\DeltavR_k}
        \leq (3 + \deltaQX) \norm{\vX_k} \leq 4 \norm{\vX_k}.
    \end{equation}
    Multiplying \eqref{eq:PIP:vRk} by $\RRbar^T_{k-1}$ from left together with \eqref{eq:PIP:res} leads to
    \begin{equation} \label{eq:lem:PIP:3.2:RTRk}
        \begin{split}
            \RRbar^T_{k-1} \RRbar_{1:k-1,k}
            & = \RRbar^T_{k-1}(\bQQbar^T_{k-1} \vX_k + \DeltavR_k) \\
            & = \bXX^T_{k-1} \vX_k + \DeltaEE_{1:k-1,k}, 
        \end{split}
    \end{equation}
    where $\DeltaEE_{1:k-1,k} := \DeltabDD^T_{k-1} \vX_k + \RRbar^T_{k-1} \DeltavR_k$.  From \eqref{eq:PIP:cholres}, \eqref{eq:PIP:res}, Lemma~\ref{lem:PIP:3.1}, and \eqref{eq:PIP:vRk} it follows that
    \begin{equation} \label{eq:lem:PIP:3.2:DeltavEk}
        \begin{split}
            \norm{\DeltaEE_{1:k-1,k}}
            & \leq 6 \cdot \rho_{k-1} \norm{\bXX_{k-1}} \norm{\vX_k} + \deltaQX \sqrt{1 + \xi_{k-1}} \norm{\bXX_{k-1}} \norm{\vX_k} \\
            & \leq \bigl(6 \cdot \rho_{k-1} + \sqrt{2} \cdot \deltaQX \bigr) \norm{\bXX_{k-1}} \norm{\vX_k}.
        \end{split}
    \end{equation}
    As for the bottom diagonal block of \eqref{eq:lem:PIP:3.2:blockR}, we can combine \eqref{eq:PIP:Pk}--\eqref{eq:PIP:chol} and rearrange terms to find
    \begin{equation}
        \begin{split}
            \RRbar^T_{1:k-1,k} \RRbar_{1:k-1,k} + \Rbar^T_k \Rbar_k
            & = \Pbar_k + \DeltaF_k + \DeltaC_k \\
            & = \vX_k^T \vX_k + \underbrace{\DeltaP_k + \DeltaF_k + \DeltaC_k}_{=: \DeltaE_k},
        \end{split}
    \end{equation}
    where
    \begin{equation} \label{eq:lem:PIP:3.2:DeltaEk}
        \norm{\DeltaE_k} \leq (\deltaXX + \deltaRR + \deltachol) \norm{\vX_k}^2.
    \end{equation}
    Combining \eqref{eq:PIP:cholres}, \eqref{eq:lem:PIP:3.2:RTRk}, \eqref{eq:lem:PIP:3.2:DeltavEk}, and \eqref{eq:lem:PIP:3.2:DeltaEk} we can write
    \begin{equation*}
        \RRbar^T_k \RRbar_k = \underbrace{\bmat{
            \bXX^T_{k-1} \bXX_{k-1}			& \bXX^T_{k-1} \vX_k \\
            \vX^T_k \bXX_{k-1}		& \vX^T_k\vX_k
        }}_{ = \bXX^T_k \bXX_k} + \underbrace{\bmat{
            \DeltaEE_{k-1}  			& \DeltaEE_{1:k-1,k} \\
            \DeltaEE^T_{1:k-1,k}	& \DeltaE_k
        }}_{ =: \DeltaEE_k},
    \end{equation*}
    where applying \cite[P.15.50]{GarH17} leads to
    \begin{align*}
        \norm{\DeltaEE_k}
        & \leq \norm{\DeltaEE_{k-1}} + 2 \norm{\DeltaEE_{1:k-1,k}} + \norm{\DeltaE_k} \\
        & \leq \xi_{k-1} \norm{\bXX_{k-1}}^2
          + 2 \bigl(6\cdot\rho_{k-1} + \sqrt{2} \cdot \deltaQX \bigr) \norm{\bXX_{k-1}} \norm{\vX_k} \\
        & \quad + (\deltaXX+ \deltachol + \deltaRR) \norm{\vX_k}^2 \\
        & \leq \bigl(\xi_{k-1} + 12 \cdot \rho_{k-1} + 2 \sqrt{2} \cdot \deltaQX + \deltaXX+ \deltachol + \deltaRR \bigr) \norm{\bXX_k}^2,
    \end{align*}
    which proves \eqref{eq:lem:PIP:3.2:cholres}.
    
    To prove \eqref{eq:lem:PIP:3.2:res}, combining \eqref{eq:PIP:Vk} and \eqref{eq:PIP:Qk} yields
    \begin{equation*}
        \vQbar_k \Rbar_k = \vX_k - \bQQbar_{k-1} \RRbar_{1:k-1,k} + \DeltavD_k,
    \end{equation*}
    where $\DeltavD_k = \DeltavV_k + \DeltavG_k$ and
    \begin{equation}\label{eq:lem:PIP:deltaY}
        \begin{split}
            \norm{\DeltavD_k}
            & \leq \deltaQR \norm{\vX_k} + \deltaQ \norm{\vQbar_k} \norm{\Rbar_k} \\
            & \leq (\deltaQR + \deltaQ) \bigl(\norm{\vX_k} + \norm{\vQbar_k} \norm{\Rbar_k} \bigr).
        \end{split}
    \end{equation}
    Writing
    \begin{equation*}
            \begin{split}
            \bQQbar_k\RRbar_k
                & = \bmat{\bQQbar_{k-1} \RRbar_{k-1} & \bQQbar_{k-1} \RRbar_{1:k-1,k} + \vQbar_k \Rbar_k} \\
                & = \underbrace{\bmat{\bXX_{k-1} & \vX_k}}_{= \bXX_k}
                  + \underbrace{\bmat{\DeltabDD_{k-1} & \DeltavD_k}}_{=: \DeltabDD_k},
            \end{split}
    \end{equation*}
    together with \eqref{eq:PIP:res}, Lemma~\ref{lem:PIP:3.1}, and \eqref{eq:lem:PIP:deltaY} gives
    \begin{align*}
        \norm{\DeltabDD_k}
        & \leq \norm{\DeltabDD_{k-1}} + \norm{\DeltavD_k} \\
        & \leq 6 \cdot \rho_{k-1} \norm{\bXX_{k-1}} + (\deltaQR + \deltaQ)(\norm{\vX_k} + \norm{\vQbar_k} \norm{\Rbar_k}) \\
        & \leq \bigl(6 \cdot \rho_{k-1} + \deltaQR + \deltaQ \bigr) \bigl(\norm{\bXX_k} + \norm{\bQQbar_k} \norm{\RRbar_k} \bigr),
    \end{align*}
    which completes the proof.
\end{proof}

\subsection{Proof of theorem \ref{thm:PIP-MP}}\label{app:thm:PIP-MP}
\begin{proof}
    We start with the base case, $k=1$.  By the assumptions \eqref{eq:thm:PIP-MP:IO:cholres} and \eqref{eq:thm:PIP-MP:IO:res}, \eqref{eq:thm:PIP-MP:cholres} and \eqref{eq:thm:PIP-MP:res} follow trivially.  Consequently, by setting $\xi_1 = \bigO{\epslo}$, $\rho_1 = \bigO{\epslo}$ in \eqref{eq:PIP:chol} and \eqref{eq:PIP:res}, respectively, and by applying \eqref{eq:PIP:rho_xi} and \eqref{eq:thm:PIP-MP:assump}, it holds that $\xi_1 \kappa^2(\vX_1) \leq \frac{1}{2}$ and $\rho_1 \kappa^2(\vX_1) \leq \frac{1}{4}$.  Then we can apply Lemma~\ref{lem:PIP:3.1} to conclude \eqref{eq:thm:PIP-MP:LOO} for $k=1$.

    Now, by following standard rounding-error analysis from \cite{Hig02} (particularly \cite[Lemma~6.6, Theorem~8.5, \& Theorem~10.3]{Hig02}), we find that for all $k \in \{2, \ldots, p\}$, there exist constants $\deltaQX, \deltaQR, \deltaXX, \deltaRR, \deltachol, \deltaQ$ such that
    \begin{equation} \label{eq:thm:PIP-MP:constants}
        \deltaQX, \deltaQR = \bigO{\epslo}
        \mbox{ and }
        \deltaXX, \deltaRR, \deltachol, \deltaQ = \bigO{\epshi}.
    \end{equation}
    and \eqref{eq:PIP:vRk}--\eqref{eq:PIP:Qk} hold.

    Consider just $k=2$ for a moment. Then Lemma~\ref{lem:PIP:3.2} can be applied (because we already have the bound \eqref{eq:thm:PIP-MP:LOO} for $k=1$) to conclude
    \begin{align*}
        \RRbar_2^T \RRbar_2
        & = \bXX_2^T \bXX_2 + \DeltaEE_2,
        \quad \norm{\DeltaEE_2} \leq \xi_2 \norm{\bXX_2}^2; \mbox{ and} \\
        \bQQbar_2 \RRbar_2
        & = \bXX_2 + \bDD_2, \quad \norm{\bDD_2} \leq \rho_2 (\norm{\bXX_2} + \norm{\bQQbar_2} \norm{\RRbar_2}),
    \end{align*}
    where $\xi_2 = \bigO{\epslo}$ and $\rho_2 = \bigO{\epslo}$ because $\xi_1 = \bigO{\epslo}$ and the bounds \eqref{eq:thm:PIP-MP:constants}.  Thus \eqref{eq:thm:PIP-MP:cholres} and \eqref{eq:thm:PIP-MP:res} hold for $k=2$, and Lemma~\ref{lem:PIP:3.1} can subsequently be applied to prove \eqref{eq:thm:PIP-MP:LOO} for $k=2$.

    Proceeding by strong induction and identical logic as for $k=2$, we can then conclude \eqref{eq:thm:PIP-MP:cholres}--\eqref{eq:thm:PIP-MP:LOO} hold for all $k \in \{2, \ldots, p\}$.
\end{proof}

%% file: main_csc.bbl
\begin{thebibliography}{10}

\bibitem{BalCDetal14}
G.~Ballard, E.~Carson, J.~Demmel, M.~Hoemmen, N.~Knight, and O.~Schwartz.
\newblock Communication lower bounds and optimal algorithms for numerical
  linear algebra.
\newblock {\em Acta Numerica 2011, Vol 20}, 23(2014):1--155, 2014.
\newblock \href {https://doi.org/10.1017/S0962492914000038}
  {\path{doi:10.1017/S0962492914000038}}.

\bibitem{BalDGetal15}
G.~Ballard, J.~Demmel, L.~Grigori, M.~Jacquelin, N.~Knight, and H.~Nguyen.
\newblock Reconstructing {{Householder}} vectors from {{Tall-Skinny QR}}.
\newblock {\em J. Parallel Distr. Com.}, 85:3--31, 2015.
\newblock \href {https://doi.org/10.1016/j.jpdc.2015.06.003}
  {\path{doi:10.1016/j.jpdc.2015.06.003}}.

\bibitem{Bar24}
J.~L. Barlow.
\newblock Reorthogonalized block classical {{Gram-Schmidt}} using two
  {{Cholesky-based TSQR}} algorithms.
\newblock {\em SIAM J. Matrix Anal. Appl.}, 45(3):1487--1517, 2024.
\newblock \href {https://doi.org/10.1137/23M1605387}
  {\path{doi:10.1137/23M1605387}}.

\bibitem{BarS13}
J.~L. Barlow and A.~Smoktunowicz.
\newblock Reorthogonalized block classical {{Gram-Schmidt}}.
\newblock {\em Numerische Mathematik}, 123:395--423, 2013.
\newblock \href {https://doi.org/10.1007/s00211-012-0496-2}
  {\path{doi:10.1007/s00211-012-0496-2}}.

\bibitem{BieLTetal22}
D.~Bielich, J.~Langou, S.~Thomas, K.~{\'S}wirydowicz, I.~Yamazaki, and E.~G.
  Boman.
\newblock Low-synch {{Gram}}--{{Schmidt}} with delayed reorthogonalization for
  {{Krylov}} solvers.
\newblock {\em Parallel Computing}, 112:102940, 2022.
\newblock \href {https://doi.org/10.1016/j.parco.2022.102940}
  {\path{doi:10.1016/j.parco.2022.102940}}.

\bibitem{ButHMetal24}
A.~Buttari, N.~J. Higham, T.~Mary, and B.~Vieubl{\'e}.
\newblock A modular framework for the backward error analysis of {{GMRES}}.
\newblock Technical Report hal-04525918, HAL science ouverte, 2024.
\newblock URL: \url{https://hal.science/hal-04525918}.

\bibitem{CarLMetal24b}
E.~Carson, K.~Lund, Y.~Ma, and E.~Oktay.
\newblock On the loss of orthogonality in low-synchronization variants of
  reorthogonalized block classical {{Gram-Schmidt}}.
\newblock E-Print arXiv:2408.10109, arXiv, 2024.
\newblock \href {https://doi.org/10.48550/arXiv.2408.10109}
  {\path{doi:10.48550/arXiv.2408.10109}}.

\bibitem{CarLR21}
E.~Carson, K.~Lund, and M.~Rozlo{\v z}n{\'i}k.
\newblock The stability of block variants of classical {{Gram-Schmidt}}.
\newblock {\em SIAM J. Matrix Anal. Appl.}, 42(3):1365--1380, 2021.
\newblock \href {https://doi.org/10.1137/21M1394424}
  {\path{doi:10.1137/21M1394424}}.

\bibitem{CarLRetal22}
E.~Carson, K.~Lund, M.~Rozlo{\v z}n{\'i}k, and S.~Thomas.
\newblock Block {{Gram-Schmidt}} algorithms and their stability properties.
\newblock {\em Linear Algebra Appl.}, 638(20):150--195, 2022.
\newblock \href {https://doi.org/10.1016/j.laa.2021.12.017}
  {\path{doi:10.1016/j.laa.2021.12.017}}.

\bibitem{CarM24}
E.~Carson and Y.~Ma.
\newblock On the backward stability of s-step {{GMRES}}.
\newblock E-Print arXiv.2409.03079, arXiv, 2024.
\newblock \href {https://doi.org/10.48550/arXiv.2409.03079}
  {\path{doi:10.48550/arXiv.2409.03079}}.

\bibitem{Car15}
E.~C. Carson.
\newblock {\em Communication-{{Avoiding Krylov Subspace Methods}} in {{Theory}}
  and {{Practice}}}.
\newblock PhD thesis, Department of Computer Science, University of California,
  Berkeley, 2015.
\newblock URL: \url{http://escholarship.org/uc/item/6r91c407}.

\bibitem{FukKNetal20}
T.~Fukaya, R.~Kannan, Y.~Nakatsukasa, Y.~Yamamoto, and Y.~Yanagisawa.
\newblock Shifted {{Cholesky QR}} for computing the {{QR}} factorization of
  ill-conditioned matrices.
\newblock {\em SIAM Journal on Scientific Computing}, 42(1):A477--A503, 2020.
\newblock \href {https://doi.org/10.1137/18M1218212}
  {\path{doi:10.1137/18M1218212}}.

\bibitem{GarH17}
S.~R. Garcia and R.~A. Horn.
\newblock {\em A {{Second Course}} in {{Linear Algebra}}}.
\newblock Cambridge University Press, Cambridge, 2017.
\newblock \href {https://doi.org/10.1017/9781316218419}
  {\path{doi:10.1017/9781316218419}}.

\bibitem{GirLRetal05}
L.~Giraud, J.~Langou, M.~Rozlo{\v z}n{\'i}k, and J.~Van Den~Eshof.
\newblock Rounding error analysis of the classical {{Gram-Schmidt}}
  orthogonalization process.
\newblock {\em Numerische Mathematik}, 101:87--100, 2005.
\newblock \href {https://doi.org/10.1007/s00211-005-0615-4}
  {\path{doi:10.1007/s00211-005-0615-4}}.

\bibitem{GolV13}
G.~H. Golub and C.~F. Van~Loan.
\newblock {\em Matrix {{Computations}}}.
\newblock Johns {{Hopkins Studies}} in the {{Mathematical Sciences}}. Johns
  Hopkins University Press, Baltimore, 4 edition, 2013.

\bibitem{Hig02}
N.~J. Higham.
\newblock {\em Accuracy and Stability of Numerical Algorithms}.
\newblock {Society for Industrial and Applied Mathematics}, Philadelphia, 2nd
  ed edition, 2002.

\bibitem{Hoe10}
M.~Hoemmen.
\newblock {\em Communication-Avoiding {{Krylov}} Subspace Methods}.
\newblock PhD thesis, Department of Computer Science, University of California
  at Berkeley, 2010.
\newblock URL:
  \url{http://www2.eecs.berkeley.edu/Pubs/TechRpts/2010/EECS-2010-37.pdf}.

\bibitem{Lun23}
K.~Lund.
\newblock Adaptively restarted block {{Krylov}} subspace methods with
  low-synchronization skeletons.
\newblock {\em Numerical Algorithms}, 93(2):731--764, 2023.
\newblock \href {https://doi.org/10.1007/s11075-022-01437-1}
  {\path{doi:10.1007/s11075-022-01437-1}}.

\bibitem{LunOCetal24}
K.~Lund, E.~Oktay, E.~Carson, and Y.~Ma.
\newblock {{BlockStab}}, 2024.
\newblock URL: \url{https://github.com/katlund/BlockStab}.

\bibitem{MorYZ12}
D.~Mori, Y.~Yamamoto, and S.~L. Zhang.
\newblock Backward error analysis of the {{AllReduce}} algorithm for
  householder {{QR}} decomposition.
\newblock {\em Japan Journal of Industrial and Applied Mathematics},
  29(1):111--130, 2012.
\newblock \href {https://doi.org/10.1007/s13160-011-0053-x}
  {\path{doi:10.1007/s13160-011-0053-x}}.

\bibitem{Okt24}
E.~Oktay.
\newblock {\em Mixed-Precision Computations in Numerical Linear Algebra}.
\newblock PhD thesis, Faculty of Mathematics and Physics, Charles University,
  Prague, 2024.
\newblock URL:
  \url{https://dspace.cuni.cz/bitstream/handle/20.500.11956/191480/140119625.pdf?sequence=1}.

\bibitem{OktC23}
E.~Oktay and E.~Carson.
\newblock Using {{Mixed Precision}} in {{Low-Synchronization Reorthogonalized
  Block Classical Gram-Schmidt}}.
\newblock {\em PAMM}, 23(1):e202200060, 2023.
\newblock \href {https://doi.org/10.1002/pamm.202200060}
  {\path{doi:10.1002/pamm.202200060}}.

\bibitem{SmoBL06}
A.~Smoktunowicz, J.~L. Barlow, and J.~Langou.
\newblock A note on the error analysis of classical {{Gram-Schmidt}}.
\newblock {\em Numerische Mathematik}, 105(2):299--313, 2006.
\newblock \href {https://doi.org/10.1007/s00211-006-0042-1}
  {\path{doi:10.1007/s00211-006-0042-1}}.

\bibitem{Ste08}
G.~W. Stewart.
\newblock Block {{Gram-Schmidt}} orthogonalization.
\newblock {\em SIAM Journal on Scientific Computing}, 31(1):761--775, 2008.
\newblock \href {https://doi.org/10.1137/070682563}
  {\path{doi:10.1137/070682563}}.

\bibitem{ThoCRetal23}
S.~Thomas, E.~Carson, M.~Rozlo{\v z}n{\'i}k, A.~Carr, and K.~{\'S}wirydowicz.
\newblock Iterated {{Gauss}}--{{Seidel GMRES}}.
\newblock {\em SIAM Journal on Scientific Computing}, pages S254--S279, 2023.
\newblock URL: \url{https://epubs.siam.org/doi/10.1137/22M1491241}, \href
  {https://doi.org/10.1137/22M1491241} {\path{doi:10.1137/22M1491241}}.

\bibitem{TreB97}
L.~N. Trefethen and D.~I. Bau.
\newblock {\em Numerical {{Linear Algebra}}}.
\newblock SIAM, Philadelphia, 1997.

\bibitem{XuAD23}
Z.~Xu, J.~J. Alonso, and E.~Darve.
\newblock A numerically stable communication-avoiding s-step {{GMRES}}
  algorithm.
\newblock Technical Report arXiv:2303.08953, arXiv, 2023.
\newblock \href {https://doi.org/10.48550/arXiv.2303.08953}
  {\path{doi:10.48550/arXiv.2303.08953}}.

\bibitem{YamNYetal15}
Y.~Yamamoto, Y.~Nakatsukasa, Y.~Yanagisawa, and T.~Fukaya.
\newblock Roundoff error analysis of the {{Cholesky QR2}} algorithm.
\newblock {\em Electronic Transactions on Numerical Analysis}, 44:306--326,
  2015.
\newblock URL:
  \url{http://www.emis.de/journals/ETNA/vol.44.2015/pp306-326.dir/pp306-326.pdf}.

\bibitem{YamHBetal24}
I.~Yamazaki, A.~J. Higgins, E.~G. Boman, and D.~B. Szyld.
\newblock Two-{{Stage Block Orthogonalization}} to {{Improve Performance}} of
  s-step {{GMRES}}.
\newblock In {\em 2024 {{IEEE International Parallel}} and {{Distributed
  Processing Symposium}} ({{IPDPS}})}, pages 26--37, San Francisco, CA, USA,
  2024.
\newblock \href {https://doi.org/10.1109/IPDPS57955.2024.00012}
  {\path{doi:10.1109/IPDPS57955.2024.00012}}.

\bibitem{YamTHetal20}
I.~Yamazaki, S.~Thomas, M.~Hoemmen, E.~G. Boman, K.~{\'S}wirydowicz, and J.~J.
  Eilliot.
\newblock Low-synchronization orthogonalization schemes for s-step and
  pipelined {{Krylov}} solvers in {{Trilinos}}.
\newblock In {\em Proceedings of the 2020 {{SIAM Conference}} on {{Parallel
  Processing}} for {{Scientific Computing}} ({{PP}})}, pages 118--128, 2020.
\newblock \href {https://doi.org/10.1137/1.9781611976137.11}
  {\path{doi:10.1137/1.9781611976137.11}}.

\bibitem{Zou23}
Q.~Zou.
\newblock A flexible block classical {{Gram}}--{{Schmidt}} skeleton with
  reorthogonalization.
\newblock {\em Numerical Linear Algebra with Applications}, 30(5):e2491, 2023.
\newblock \href {https://doi.org/10.1002/nla.2491}
  {\path{doi:10.1002/nla.2491}}.

\end{thebibliography}
